\documentclass[11pt,reqno,twoside]{amsart}
\usepackage{amsmath,amsthm, amscd, amssymb, amsfonts, palatino}
\usepackage[utf8]{inputenc}
\usepackage{color}
\usepackage[all]{xy}
\usepackage{tikz}
\usetikzlibrary{matrix,arrows,decorations.pathmorphing}
\usepackage{tikz-cd}
\usepackage{hhline}
\newcommand{\A}{{\mathcal A}}

\newcommand{\G}{{\mathcal G}}

\newcommand{\F}{{\mathcal F}}
\newcommand{\C}{{\mathcal C}}
\newcommand{\D}{{\mathcal D}}
\newcommand{\Ec}{{\mathcal E}}
\newcommand{\E}{{\mathbf E}}
\newcommand{\Uc}{{\mathcal U}}

\newcommand{\B}{{\mathcal B}}
\newcommand{\T}{{\mathcal T}}
\newcommand{\Hc}{{\mathcal H}}
\newcommand{\Vc}{{\mathcal V}}
\newcommand{\Wc}{{\mathcal W}}
\newcommand{\Pc}{{\mathcal P}}
\newcommand{\Kc}{{\mathbf K}}
\newcommand{\ep}{{\epsilon}}

\newcommand{\Sc}{{\mathbb S}}

\DeclareMathOperator*{\Tim}{\times}
\newcommand{\Times}[2]{\sideset{_#1}{_#2}\Tim}
\newcommand{\ubs}[1]{{#1}^{\bullet\bullet}}
\newcommand{\dbs}[1]{{#1}_{\bullet\bullet}}
\newcommand{\pfib}[2]{\;{}_{#1}\hspace{-2pt}\times_{#2}\;}

\newcommand{\Ext}{\operatorname{Ext}}

\newcommand{\Ho}{\operatorname{H}}

\newcommand\Hom{\operatorname{Hom}}
\newcommand\Opext{\operatorname{\mathcal{O}pext}}
\newcommand\Tot{\operatorname{Tot}}

\newcommand\pfibrado[2]{\;{}_{#1}\hspace{-2pt}\times_{#2}\;} 

\numberwithin{equation}{section}\theoremstyle{plain}

\newtheorem{theorem}{Theorem}[section]
\newtheorem{lemma}[theorem]{Lemma}

\newtheorem{proposition}[theorem]{Proposition}

\theoremstyle{definition}
\newtheorem{definition}[theorem]{Definition}
\newtheorem{exa}[theorem]{Example}

\theoremstyle{remark}
\newtheorem{obs}[theorem]{Remark}
\newtheorem{remark}[theorem]{Remark}

\newcommand\id{\operatorname{id}}
\newcommand\idd{\mathbf{id}}

\newcommand\iddv{\mathbf{id}}

\def\pf{\begin{proof}}
\def\epf{\end{proof}}

\theoremstyle{remark}

 \newcommand\caja[5]{
 \begin{tabular}{c}
 \begin{picture}(50,60)(0,0) 
 \put(5,8){\framebox(40,40){#1}}
 \put(25,55){\makebox(0,0)[c]{${\scriptstyle \text{#2}}$}} 
 \put(48,28){\makebox(0,0)[l]{${\scriptstyle \text{#3}}$}} 
 \put(25,0){\makebox(0,0)[c]{${\scriptstyle \text{#4}}$}}  
 \put(3,28){\makebox(0,0)[r]{${\scriptstyle \text{#5}}$}} 
 \end{picture}
 \end{tabular}
 }

\newcommand\cajaMedium[5]{
 \begin{tabular}{c}
 \begin{picture}(37,50) 
 \put(3,8){\framebox(30,30){#1}}
 \put(18,43){\makebox(0,0)[c]{${\scriptstyle \text{#2}}$}} 
 \put(37,23){\makebox(0,0)[l]{${\scriptstyle \text{#3}}$}} 
 \put(18,1){\makebox(0,0)[c]{${\scriptstyle \text{#4}}$}}  
 \put(1,23){\makebox(0,0)[r]{${\scriptstyle \text{#5}}$}} 
 \end{picture}
 \end{tabular}
 }

\title [Double Groupoid Cohomology]{Double groupoid cohomology and 
extensions}

\author[Ochoa]{Jes\'us Alonso Ochoa Arango*}
\email{jesus.ochoa@javeriana.edu.co}
\author[Tiraboschi]{Alejandro Tiraboschi**}
\email{tirabo@famaf.unc.edu.ar}

\address{\noindent
*Departamento de Matem\'aticas, Facultad de ciencias,
Pontificia Universidad Javeriana. Bogot\'a, Colombia.
}

\address{\noindent
** Facultad de Matem\'atica, Astronom\'\i a y F\'\i sica, Universidad
Nacional de C\'ordoba.  CIEM -- CONICET. (5000) Ciudad
Universitaria, C\'ordoba, Argentina.}

\date{\today}

\begin{document}

\renewcommand{\baselinestretch}{1.2}
\thispagestyle{empty}

\begin{abstract}
We study extensions of double groupoids in the sense of \cite{AN2}
and show some classical results of group theory extensions
in the case of double groupoids. For it, given a double
groupoid $(\B; \Vc,\Hc; \Pc)$ \emph{acting} on an abelian group bundle 
$\Kc \to \Pc$, we introduce a cohomology double complex, in a similar way as was done
in \cite{AN2} and we show that the extensions of $\F$ by $\Kc$ are classified
by the total first cohomology group of the associated total complex.

With the aim to extend the above results to the topological setting,
following ideas of Deligne \cite{D} and Tu \cite{tu}, 
by means of simplicial methods, we introduce a \emph{sheaf cohomology} 
for topological double groupoids, generalizing the double groupoid 
cohomological in the discrete case, and we carry out in the topological setting 
the results obtained for discrete double groupoids.
\end{abstract}

\maketitle

\section*{Introduction}

The notion of double groupoids was introduced by Ehresmann \cite{ehr},
and later studied in \cite{brown, bj, BM, bs} and references therein.

A double groupoid is a set $\B$ endowed with two different
but compatible groupoid structures.
It is useful to represent the elements of $\B$ as boxes that merge
horizontally or vertically according to the
groupoid multiplication into consideration.
The vertical (respectively horizontal) sides of a box belong
to another groupoid $\Vc$ (resp. $\Hc$). Later, the notion of double Lie groupoid 
was defined and investigated by K. Mackenzie \cite{mk1, mk2}; see also \cite{p, mk3, wl}
for applications to differential and Poisson geometry.
In particular the question of the classification of double
Lie groupoids was raised in \cite{mk1}, see also \cite{BM}.
Since then, several classes of double groupoids has been classified;
in \cite{BM}, a complete answer was given in the restricted
case of locally trivial double Lie groupoids.  More recently in \cite{AN3},
was given a description in two stages of discrete double groupoids.
Before to state it, we remind that a diagram
of groupoids over a pair of groupoids $\Vc$ and $\Hc$ is a triple $(\D, j, i)$
where $\D$ is a groupoid  and $i: \Hc \to \D$, $j: \Vc
\to \D$ are morphisms of groupoids (over a fixed set of points).
In  \cite{AN3} the authors showed that:
\begin{itemize}
\item {\bf Stage 1. \cite[Thm. 1.9]{AN3}. } \label{stage_1}
Any double groupoid is an extension of a slim double groupoid (its \emph{frame}) by an abelian group bundle.
\end{itemize}
In the same paper, the authors also prove the following theorem.
\begin{itemize}
\item {\bf Stage 2. \cite[Thm. 2.8]{AN3}. } \label{stage_2}
The category of slim double groupoids, with fixed vertical
and horizontal goupoids $\Vc$ and $\Hc$, satisfying the filling condition,
is equivalent to the category of diagrams over $\Vc$ and $\Hc$.
\end{itemize}

In \cite{AOT},  we extend theorem \ref{stage_1} to the setting of double Lie
groupoids. In this context, the usual filling condition is replaced
by the request that the \emph{double source map} is a surjective submersion \cite{mk1}.
As is expected, theorem \ref{stage_1} doesn't has an identical
statement in this case, and there are some topological and
geometrical ingredients to be taken in account. The main result of that work was,
\begin{itemize}
\item \cite[Thm. 3.7]{AOT}
\emph{The category of slim double Lie groupoids, with fixed vertical
and horizontal Lie groupoids $\Vc$ and $\Hc$, and proper core action,
is equivalent to the category of diagrams of Lie groupoids $(\D, j, i)$
such that the maps $j$ and $i$ are transversal at the identities.}
\end{itemize}

In this work we extend an equivalent formulation of {\bf stage 2} to the  setting of double topological groupoids. The paper is divided in two parts and in what follows we will describe the contents of each one of them.

The first part goes from section \ref{prels} to \ref{seccion4} and develop the cohomology of discrete double groupoids. In the section \ref{prels} we remind some basic facts about double groupoids (discrete and topological).  In section \ref{discrete_cohomology_section}, we introduce the cohomology of double groupoids by means of the total complex of a double complex attached to any double groupoid. In section \ref{seccion4} 
we show that given a double groupoid $(\F; \Vc, \Hc; \Pc)$ acting on an abelian group bundle $\Kc \to \Pc$, all the extensions of $\F$ by $\Kc$ are classified by the first total cohomology group of such total complex (Thm. \ref{corresp-cocycles-opext}), i. e. there is a bijection between $\mathcal{O}pext(\F,\Kc)$
and the cohomology group $\Ho^1_{\Tot}(\F, \Kc)$. This result is basically an extension of Stage 2. 

The second part of this work, sections 4 and 5, is concerned with a geometric construction of the bisimplicial set associated to a double groupoid and with the simplicial cohomology of topological double groupoids. 
Secction 4 contains basic facts about simplicial and bisimplicial sets. The main result of this section is Thm. \ref{double_skeleton_core_groupoid} that provides a construction of the double skeleton of a double groupoid as a homogeneous space of the core groupoid attached to the double groupoid. This result enables us to realize the double skeleton as a bisimplicial set (or space, or manifold, depends on the context) in a simple way.

In section 5  we extend the cohomology theory developed in the first part to the setting of topological double groupoids. For the discrete case, the proof of proposition \ref{any.ext.is.smash} (a reformulation of Stage 2), depends strongly on the existence of a section of the frame map that maps a double groupoid onto its frame. In the continuous (differentiable) setting we cannot guarantee any more the existence of a continuous (resp. smooth) global section of this map. If the frame of a double topological (resp. Lie) groupoid is a topological space (resp. smooth manifold) and the frame map is an open surjective map  (resp.  a surjective submersion), we can assure the existence of local sections and then, we can try to localize the extension process to open sets where the local sections exists. To do this, in this section we develop a \v{C}ech double groupoid cohomology that will allow us to classify the extensions of topological double groupoids by abelian topological group bundles in a similar way that in the discrete case (Thm. \ref{class_ext_simplicial_version}). 
Roughly speaking, given a \emph{bisimplicial sheaf} over a \emph{bisimplicial set}, we can define a double complex associated to it, and then we could define the cohomology of the bisimplicial set with values in the bisimplicial sheaf as the cohomology of the associated total complex:  let $(\F; \Vc, \Hc; \Pc)$ be a  topological double  groupoid  acting on a bundle $\Kc \to \Pc$ of topological abelian groups. If $\mathcal{U} = \{ \mathcal{U}_i \}_{i \in I}$ is an open cover of $\Pc$, the double groupoid $(\F[\mathcal{U}]); \Vc[\mathcal{U}], \Hc[\mathcal{U}]; \Pc[\mathcal{U}])$
is the \textit{\v{C}ech double groupoid associated to $(\F, \mathcal{U})$} (see def.  \ref{cech_groupoid} and \ref{open_cover}), and $\Opext(\F[\mathcal{U}], \Kc[U])$ denotes the set of all extensions of $\F[\Uc]$ by $\Kc[\Uc]$. Now, we can construct on the \emph{double skeleton} $\ubs{\F} = \{ \F^{(m,n)} \}_{(m,n) \in \mathbb{N}^2}$ of the double groupoid, a double bisimplicial sheaf $\ubs{\A}$  associated with the action $\Kc \to \Pc$ and  define 
\begin{equation}\label{extension_form_0}
\Ext(\F, \Kc) := 	\underset{\mathcal{U}}  {\underset{\longrightarrow} \lim } \; \Opext(\F[\mathcal{U}], \Kc[U]), 
\end{equation}
where  $\mathcal{U} = \{ \mathcal{U}_i \}_{i \in I}$ runs over open
covers of $\Pc$. Then,  the main result of the paper could be stated as
\begin{equation}
\Ext(\F, \Kc) \cong \check{\Ho}^{1}_{\Tot}(\dbs{\Uc}; \ubs{\A}).
\end{equation}

\renewcommand{\baselinestretch}{1.2}
\renewcommand{\thefootnote}{}
\thispagestyle{empty}

\begin{section}{Preliminaries on groupoids an double groupoids}\label{prels}

We denote a groupoid in the form $ \xymatrix{\G \ar@<2pt>[r]^{s}
\ar@<-2pt>[r]_{e} &\Pc}$, where $s$ and $e$ stand for \emph{source} and
\emph{end} respectively; and the identity map will be denoted by $\id: \Pc \to \G$.

Recall that a groupoid $ \xymatrix{\G \ar@<2pt>[r]^{s}
\ar@<-2pt>[r]_{e} & \Pc} $ is a \textit{topological groupoid} \cite{renault} if
$\Pc$ and $\G$ are topological spaces and all the structural maps are continuous. 
The {\em anchor} of $\G$   is the map $\chi:\G \to \Pc \times \Pc$ given by
$\chi(g) = (s(g),e(g) )$.
We recall the following well known definition.

\begin{definition}
A {\em left action} of a groupoid $ \xymatrix{\G \ar@<2pt>[r]^{s}
\ar@<-2pt>[r]_{e} &\Pc}$ {\em along} a map 
$\epsilon:\Ec \to \Pc$ is given by a map from the pullback $\G \pfibrado{e}{\epsilon}\Ec$ to $\Ec$, denoted by $(g,n) \mapsto gn$, such that:
$$\epsilon(hy) = s(h),\quad \id(\epsilon(y))\; y = y, \quad (gh)y = g(hy),$$
for all $g,h \in \G$ and $y \in N$ such that $e(g) = s(h)$ and $e(h)
= \epsilon(y)$. We shall simply say that the groupoid $\G$ acts on $\Ec$.
\end{definition}

\begin{definition}[Ehresmann]

A \textit{double groupoid} is a groupoid object internal to the
category of groupoids. In other terms, a \textit{double groupoid} consists
of a set $\B$ with two groupoid structures with \textit{bases} $\Hc$
and $\Vc$, which are themselves groupoids over a common base $\Pc$,
all subject to the compatibility condition that the structure maps
of each structure are morphisms with respect to the other.
\end{definition}

It is usual to represent a double groupoid $(\B; \Vc, \Hc; \Pc)$
in the form of a diagram of four related groupoids
$$
\xymatrix{ \B \ar@<2pt>[rr]^{l} \ar@<-2pt>[rr]_{r} \ar@<2pt>[d]^{b}
\ar@<-2pt>[d]_{t}
& & \Vc \ar@<2pt>[d]^{b} \ar@<-2pt>[d]_{t}\\
\Hc  \ar@<2pt>[rr]^{l} \ar@<-2pt>[rr]_{r} & & \Pc  }
$$
where $t$, $b$, $l$, $r$ mean \emph{``top'', ``bottom'', ``left''} and \emph{``right''},
respectively. We sketch the main axioms that these groupoids should
satisfy and refer \emph{e.~g.} to \cite[Sec. 2]{AN1} and
\cite[Sec. 1]{AN2} for a detailed exposition and other
conventions.

The elements of $\B$ are called \emph{``boxes''} and will be denoted by
$$
A = \quad\caja{$A$}{$t(A)$}{$r(A)$}{$b(A)$}{$l(A)$}\quad \in\B.
$$
Here $t(A),\;b(A) \in \Hc$ and $l(A),\;r(A) \in \Vc$. The identity
maps will be denoted $\idd: \Vc \to \B$ and $\idd: \Hc \to \B$. The
product in the groupoid $\B$ with base $\Vc$ is called {\em
horizontal product} and is denoted by $AB$ or $\{AB\}$, for $A,B \in \B$ with
$r(A) = l(B)$. The product in the groupoid $\B$ with base $\Hc$ is
called {\em vertical product} and is denoted by $\begin{matrix}
A\\B\end{matrix}$ or $\left\{\begin{matrix}
A\\B\end{matrix}\right\}$, for $A,B \in \B$ with $b(A) =t(B)$. This
pictorial notation is useful to understand the products in the
double structure. For instance, compatibility axioms between the
horizontal and vertical products with respect to source and target maps
of the horizontal and vertical groupoid structures on $\B$, are described by

$$
\cajaMedium{$A$}{$t$}{$r$}{$b$}{$l$} \;
\cajaMedium{$B$}{$t'$}{$r'$}{$b'$}{$r$} =
\cajaMedium{$\{AB\}$}{$tt'$}{$r'$}{$bb'$}{$l$}\quad \text{ and }
\quad
\begin{matrix}
\cajaMedium{$A$}{$t$}{$r$}{$b$}{$l$} \\
\cajaMedium{$B$}{$b$}{$r'$}{$b'$}{$l'$}
\end{matrix} =\;
\cajaMedium{$\scriptstyle{\left\{\begin{matrix} A
\\B\end{matrix}\right\}}$}{$t$}{$rr'$}{$b'$}{$ll'$}\quad.
$$
We omit the letter inside the box if no confusion arises. We also
write $A^h$  and $A^v$ to denote the inverse of $A\in \B$ with
respect to the horizontal and vertical structures of  groupoid over
$\B$ respectively. When one of the sides of a box is an identity, we
draw this side as a double edge. For example, if
$t(A) = \id_p$, we draw \begin{tabular}{|p{0,1cm}|} \hhline{|=|} \\
\hline\end{tabular}  and say that  $t(A) \in \Pc$.

\noindent \textbf{The interchange law:} The most important axiom of double groupoids is undoubtedly the \emph{interchange law}, 
it states that
\begin{equation}\label{interchange_law}
\begin{matrix} 
\left\{ K L \right\} \vspace{-2pt}\\ \left\{ M N \right\} 
\end{matrix} 
=
\left\{ 
\begin{matrix} K  \vspace{-4pt} \\ M   
\end{matrix}
\right\}
\left\{
\begin{matrix} 
L \vspace{-4pt}\\ N  
\end{matrix}
\right\},
\end{equation}
when the four boxed are compatible i.e, when all the compositions involved
are well defined.

\begin{definition}
A double groupoid is a \textit{topological double groupoid} if all the
four groupoids involved are topological groupoids and the \textit{top-right
corner map}
\begin{equation}\label{t-r corner map}
\urcorner: \B \to \Hc \pfibrado{l}{t} \Vc, \qquad A \mapsto
\urcorner(A) = (t(A), r(A)),
\end{equation}
is a continuous surjective map.
\end{definition}

We shall say that a double groupoid is \emph{discrete} if no differentiable or topological structure is present.
A  discrete double groupoid satisfies the \textit{filling condition}
if the the \textit{top-right corner map} defined in \eqref{t-r corner map}
is surjective. We refer the reader to \cite{AN3} for more details on discrete
double groupoids and the filling condition.

Let $(\B;\Vc,\Hc;\Pc)$ be a double groupoid, if $P \in \Pc$, we denote  $  \Theta_P := \idd \circ \id(P)$.

\begin{definition}[Brown and Mackenzie] Let $(\B;\Vc,\Hc;\Pc)$ be a double groupoid. The core
groupoid $\E(\B)$ of $\B$  is the set
$$\E(\B)=\{E \in \B : \; t(E), \;
r(E) \in \Pc\}$$ 
with source and target projections
$s_{{}_\E}$, $e_{{}_\E}: \E(\B) \to \Pc$, given by
$s_{{}_\E}(E)=bl(E)$ and $e_{{}_\E}(E)= tr(E)$ respectively;
identity map given by $\Theta_P$; multiplication and inverse given by
\begin{equation}\label{productcore} E \circ F : =
\left\{\begin{matrix} {\scriptstyle \idd l(F)} & F\\
E & {\scriptstyle \idd(b(F)) \vspace{-1pt}} \end{matrix} \right\},
\qquad E^{(-1)}: = (E \iddv b(E)^{-1})^v
= \left\{\begin{matrix}\iddv l(E)^{-1} \vspace{-4pt}\\
E^h\end{matrix} \right\},
\end{equation}
for every compatible $E , F\in \E(\B)$. 

We observe that
the elements of $\E(\B)$ are of the form $E =
\begin{tabular}{|p{0,1cm}||} \hhline{|=||} \\ \hline \end{tabular}
\;$; the source gives the bottom-left vertex and the target  gives the
top-right vertex of the box. 
\end{definition}

\begin{remark}
It is clear that if $(\B;\Vc,\Hc;\Pc)$ is a topological double groupoid then its core groupoid also is. Moreover, in the case of double Lie groupoids, the definition request that the top-right corner map be a surjective submersion, then in this case the
core groupoid $\E(\B)$ is a closed embedded submanifold of $\B$ and clearly  $s_{{}_\E}$ and $e_{{}_\E}$ are
surjective submersions. With these structural maps
$\E(\B)$ becomes a Lie groupoid with base $\Pc$ (differentiability conditions being easily verified because $\E(\B)$
is an embedded submanifold of $\B$).
\end{remark}

Another important invariant of a double groupoid is the intersection
$\Kc(\B)$ of all four core groupoids:
\begin{align*} \Kc(\B) & : = \{K \in
\B: \; t(K), b(K),l(K), r(K)  \in \Pc \}.\end{align*}
Thus a box is in $\Kc(\B)$ if and only if it is  of the form \begin{tabular}{||p{0,1cm}||} \hhline{|=|} \\
\hhline{|=|}\end{tabular}. 
Let $p: \Kc(\B) \to \Pc$ be the `common
vertex' function, say $p(K) = lb(K)$. For any $P\in \Kc(\B)$, let
$\Kc(\B)_P$ be the fiber  at $P$; $\Kc(\B)_P$ is an abelian group under
vertical composition, that coincides with horizontal composition. This
is just the well-known fact: ``a double group is the same as an
abelian group". Indeed, apply the interchange law
\begin{equation*}
\begin{matrix} (K L) \vspace{-2pt}\\ (M N) \end{matrix} =
\left(\begin{matrix} K  \vspace{-4pt}\\ M   \end{matrix}\right)
\left(\begin{matrix} L \vspace{-4pt}\\ N  \end{matrix}\right)
\end{equation*}
to four boxes $K,L,M,N \in \Kc_P$: if $L = M = \Theta_P$, this
says that $\begin{matrix} K  \vspace{-4pt}\\ N \end{matrix}  = KN$
and the two operations coincide.  If, instead, $K = N = \Theta_P$,
 this says that $\begin{matrix} L  \vspace{-4pt}\\ M \end{matrix}
= ML$, hence the composition is abelian. Note that this operation
in $\Kc(\B)$ coincides also with the core multiplication
\eqref{productcore}. In short, $\Kc(\B)$ is an abelian group bundle
over $\Pc$.

Once an algebraic structure is introduced the following natural task
is to study the class of objects on which they act. There are
several notion of \emph{modules} for double groupoids \cite{aa, BM},
here we introduce a new one implicit in \cite{AN3}.

\begin{definition}\label{double groupoid left action}
Let $\T = (\F; \Vc, \Hc; \Pc)$ be a double groupoid and let $\epsilon: \Ec \to \Pc$
be a map. We say that {\em $\mathcal{T}$ acts on the left along $\epsilon$}
if the following conditions holds
\begin{enumerate}
\item The groupoid $\xymatrix{ \Vc \ar@<-2pt>[r]_r \ar@<2pt>[r]^l & \Pc}$
acts on $\Ec \to \Pc $ on the left,
\item The groupoid $\xymatrix{ \Hc \ar@<-2pt>[r]_b \ar@<2pt>[r]^t & \Pc}$
acts on $\Ec \to \Pc $ on the left,
\item for any box $A \in \F$ the following relation holds
\begin{equation}\label{acciondoble}
l(A)^{-1}\cdot(t(A)\cdot X) = b(A)\cdot (r(A)^{-1}\cdot X)
\end{equation}
for any $X$ where the action is defined.
\end{enumerate}
\end{definition}

\begin{remark}
From the above definition it is easy to see that
any module over a double groupoid is, at the same time,
a module over the core groupoid associated to it.

\noindent The notion of action introduced here, 
differs substantially from the previous ones introduced by Brown and Mackenzie \cite{BM} 
and by Andruskiewitsch and Aguiar \cite{aa}. Nevertheless it is good enough to obtain 
meaningful information of a double groupoid  from 
its \textit{representation category}. We study systematically this concept 
of representation in a forthcoming paper	.
\end{remark}

\begin{exa}
Given a double groupoid $(\F; \Hc, \Vc;\Pc)$, the vertical and the horizontal groupoids 
$\Vc$ and $\Hc$ acts on $\Kc(\B)$ by vertical, respectively horizontal, conjugation:
\begin{align}
\label{accionver}  \text{If } g\in \Vc(Q,P) \text{ and } A\in \Kc_P \text{ then }
g\cdot A &:=  \begin{matrix} \id g\vspace{-1pt}\\A  \vspace{-4pt}\\ \id g^{-1} \end{matrix}
\in \Kc_Q;&
\\
\label{accionhor} \text{if } x\in \Hc(Q,P) \text{ and } A\in \Kc_P \text{ then }
x\cdot A &:=  \id xA\id x^{-1}\in \Kc_Q;&
\end{align}
we can note that both actions are by group bundle automorphisms. This maps define an 
action of the double groupoid $\F$ on the abelian group bundle $\Kc(\F)$ associated to it.
\end{exa}

\begin{remark}
In an analogous way to \ref{double groupoid left action} we can define a right action of a double groupoid
along a map $\epsilon: \Ec \to \Pc$. The conditions $(1)$ and $(2)$ in
definition \ref{double groupoid left action} suffers the obvious modifications
and the condition three is changed by
\begin{equation}
(X \cdot l(A)) \cdot t(A)^{-1} = (X \cdot b(A)^{-1})\cdot r(A),
\end{equation}
whenever $\epsilon(X) = bl(A) = lb(A)$.
\end{remark}

\end{section}

\begin{section}{Cohomology of double groupoids}\label{discrete_cohomology_section}
Let $\T = (\F; \Vc, \Hc; \Pc)$ be a double groupoid and $p: \Kc(\F) \to \Pc$ be the associated abelian group bundle. We also denote $\Kc = \Kc(\F)$.

\vskip .3cm

\subsection{The double complex associated to a Double Groupoid}

Let $\{ \mathcal{F}^{(m,n)} \}$ be  the family of sets defined by	
\begin{align}\label{bsp-set-db-gpd}
\F^{(0,0)} &:= \Pc, \nonumber\\
\F^{(0,s)} &:=  \left\{\left(x_1, \dots, x_s\right) \in  \Hc^{\, s}:
x_1\vert x_2 \dots \vert x_s \right\} = \Hc^{(s)}, \quad s> 0, \nonumber\\
\F^{(r,0)} &:= \left\{\left(g_1, \dots, g_r\right) \in\Vc^{\, r}:
g_1\vert g_2 \dots \vert g_r\right\} = \Vc^{(r)}, \quad r> 0, \nonumber\\
\F^{(r,s)} &:= \left\{\left(\begin{tabular}{p{0,8cm} p{0,8cm} p{0,8cm} p{0,8cm}}
$A_{11}$ & $ A_{12}$ & \dots & $A_{1s}$ \\
$A_{21}$ & $ A_{22}$ & \dots & $A_{2s}$ \\
\dots & \dots & \dots & \dots \\
$A_{r1}$ & $ A_{r2}$ & \dots & $A_{rs}$  \end{tabular}\right) \in
\F^{r\times s}: \quad
\begin{tabular}{p{0,8cm}|p{0,8cm}|p{0,8cm}|p{0,8cm}}
$A_{11}$ & $ A_{12}$ & \dots & $A_{1s}$ \\ \hline
$A_{21}$ & $ A_{22}$ & \dots & $A_{2s}$ \\ \hline
\dots & \dots & \dots & \dots \\ \hline
$A_{r1}$ & $ A_{r2}$ & \dots & $A_{rs}$  \end{tabular} \right\}
\; r,s> 0;
\end{align}
that is, $\mathcal{F}^{(m,,n)}$ is the set of matrices of
composable boxes of $\F$ of size $m \times n$.

We define $D^{\cdot, \cdot} = D^{\cdot, \cdot}(\T, \Kc)$ in the following way: let $r,s>0$, then
\begin{align*}
  D^{0,0}&:= \left\{ \alpha: \Pc \to \Kc \; : \; p \circ \alpha = Id_{\Pc} \right\} \\
  D^{r,0}&:= \left\{ \alpha: \Vc^{(r)} \to \Kc\; : \; p \circ \alpha (f) = b (f_r)\text{, }p \circ \alpha (r(A_{11}),\cdots,r(A_{r1})) = br(A_{11})\text{ and }\alpha\text{ is normalized}\right\},	\\
  D^{0,s}&:= \left\{ \alpha: \Hc^{(s)} \to \Kc\; : \; p \circ \alpha (x) = l (x_1)\text{, }p \circ \alpha (t(A_{11}),\cdots,t(A_{1s})) = tl(A_{11})\text{ and }\alpha\text{ is normalized}\right\},	\\
  D^{r,s}&:= \left\{ \alpha: \F^{(r,s)}\to \Kc\; : \; p \circ \alpha(A) = bl(A_{r1})\text{ and } \alpha \text{ is normalized} \right\}.
\end{align*}
\begin{remark}
Recall that if $r,s \ge 0$, $\alpha$ is normalized if $\alpha(A) = 0$ (i.e $=\Theta_P$, for some $P\in \Pc$) when
\begin{align*}
 &r > 1, s>0  \text{ and } A_{ij} \in \Vc, \text{ some } i,j \text{, or  }\\
 &r > 1, s = 0  \text{ and } A_{i0} \in \Pc, \text{ some } i \text{, or  } \\
\ &r>0, s > 1  \text{ and } A_{ij} \in \Hc, \text{ some } i,j \text{, or  }\\
 &r =0, s > 1  \text{ and } A_{0j} \in \Pc, \text{ some } j.
\end{align*}
\end{remark}

In particular we have
\begin{equation*}
D^{1, 2}  = \left\{\alpha: \F^{(1,2)} \to \Kc | \; (p \circ \alpha)(A_{11},A_{12})
= bl(A_{11}),  \; \alpha(A_{11},A_{12}) = 0,  \text{ if } A_{11} \text{ or } A_{12} \in \Hc
\right\}.
\end{equation*}
and
\begin{equation*}
D^{2, 1}  = \left\{\alpha: \F^{(2,1)} \to \Kc | \;
(p \circ \alpha)\left(\begin{matrix}A_{11}\\A_{21}\end{matrix}\right)
= bl(A_{21}),  \; \alpha \left(\begin{matrix}A_{11}\\A_{21}\end{matrix}\right) = 0,
\text{ if } A_{11}  \text{ or }  A_{12} \in \Vc \right\}.
\end{equation*}

Let $d_H = d_H^{r,s}: D^{r, s} \to D^{r, s+1}$ and
$d_V = d_V^{r,s}: D^{r, s} \to D^{r+1, s}$
be, respectively, the horizontal and vertical coboundary
maps defined as follows:

\begin{itemize}
\item If  $r =0$, $d_H$ is

\begin{equation}\label{cob1}
\begin{aligned}d_H^{0,0} \alpha(x) &= x \cdot \alpha(r(x)) - \alpha(l(x)), \\
d_H^{0,s} \alpha(x_1, \dots, x_{s+1}) &=  x_1 \cdot \alpha(x_2, \dots, x_{s+1})
+ \sum_{1\le i \le s}(-1)^{i} \alpha(x_1, \dots, x_ix_{i+1}, \dots, x_{s+1})
\\ & \qquad + (-1)^{s+1} \alpha(x_1, \dots, x_{s}).
\end{aligned}
\end{equation}

\item if  $s =0$, $d_V$ is
\begin{equation}\label{cob2}
\begin{aligned}d_V^{0,0} \alpha(x) &= \alpha(b(x)) - x^{-1}\cdot \alpha(t(x)), \\
d_V^{r,0} \alpha(x_1, \dots, x_{r+1}) &=  \alpha(x_2, \dots, x_{r+1})
+ \sum_{1\le i \le r}(-1)^{i} \alpha(x_1, \dots, x_ix_{i+1}, \dots, x_{r+1})
\\ & \qquad + (-1)^{r+1}  x_{r+1}^{-1} \cdot \alpha(x_1, \dots, x_{r})
\end{aligned}
\end{equation}

\item if  $r =0$, $s > 0$,
$$d_V^{0,s} \alpha(A_{11}, \dots, A_{1s}) =
\alpha(b(A_{11}), \dots, b(A_{1s})) - l(A_{11})^{-1} \cdot \alpha(t(A_{11}), \dots, t(A_{1s}));$$

\item if  $r >0$, $s = 0$,
$$d_H^{r,0} \alpha\begin{pmatrix}
A_{11}   \\
\vdots   \\
A_{r1}  \end{pmatrix} =
b(A_{r1}) \cdot \alpha(r(A_{11}), \dots, r(A_{r1})) - \alpha(l(A_{11}), \dots, l(A_{r1}));$$

\item if $r > 0$ and $s = 1$, 
\begin{equation*}
d_V^{r,1} f\begin{pmatrix}
A_{11}  \\
\vdots   \\
A_{r1}  \\
A_{r+1,1}
\end{pmatrix} =
f\begin{pmatrix}
A_{21}   \\
\vdots   \\
A_{r1}  \\
A_{r+1,1}
\end{pmatrix}
+ \sum_{1\le i \le r}(-1)^{i} f\begin{pmatrix}
A_{11}  \\
\vdots  \\
\left\{\begin{matrix}A_{i1} \\A_{i+1,1}\end{matrix}\right\} \\
\vdots  \\
A_{r+1,1}
\end{pmatrix}
 + (-1)^{r+1} l(A_{r+1,1})^{-1} \cdot f\begin{pmatrix}
A_{11} \\
\vdots \\
A_{r1}  \end{pmatrix};
\end{equation*}

\item if $r = 1$ and $s > 0$, 
\begin{align*}
d_H^{1,s} f\begin{pmatrix}
A_{11}  & \dots  & A_{1,s+1}\end{pmatrix} &=
b(A_{11}) \cdot f\begin{pmatrix}
A_{12}  & \dots & A_{1,s+1}
\end{pmatrix}
\\ & \qquad + \sum_{1\le j \le s}(-1)^{j}
f\begin{pmatrix}
A_{11}  & \dots & \left\{A_{1,j} A_{1,j+1} \right\}& \dots & A_{1,s+1}
\end{pmatrix}
\\ & \qquad + (-1)^{s+1} f\begin{pmatrix}
A_{11}  & \dots & A_{1s}  \end{pmatrix}.
\end{align*}

\item More generally if $r > 0$ and $s > 0$,
\begin{align*}
d_V^{r,s} f\begin{pmatrix}
A_{11}  & \dots & A_{1s} \\
\dots  & \dots & \dots \\
A_{r1}  & \dots & A_{rs} \\
A_{r+1,1}  & \dots & A_{r+1,s}
\end{pmatrix} &=
 f\begin{pmatrix}
A_{21}  & \dots & A_{2s} \\
\dots  & \dots & \dots \\
A_{r1}  & \dots & A_{rs} \\
A_{r+1,1}  & \dots & A_{r+1,s}
\end{pmatrix}
\\ & \qquad + \sum_{1\le i \le r}(-1)^{i} f\begin{pmatrix}
A_{11}  & \dots & A_{1s} \\
\dots  & \dots & \dots \\
\left\{\begin{matrix}A_{i1} \\A_{i+1,1}\end{matrix}\right\}
& \dots & \left\{\begin{matrix}A_{is} \\A_{i+1,s}\end{matrix}\right\} \\
\dots  & \dots & \dots \\
A_{r+1,1}  & \dots & A_{r+1,s}
\end{pmatrix}
\\ & \qquad + (-1)^{r+1} l(A_{r+1,1})^{-1} \cdot f\begin{pmatrix}
A_{11}  & \dots & A_{1s} \\
\dots  & \dots & \dots \\
A_{r1}  & \dots & A_{rs} \end{pmatrix};
\end{align*}

\begin{align*}
d_H^{r,s} f\begin{pmatrix}
A_{11}  & \dots  & A_{1,s+1} \\
\dots  & \dots & \dots \\
A_{r1}  & \dots & A_{r,s+1} \end{pmatrix} &=
b(A_{r1})  \cdot f\begin{pmatrix}
A_{12}  & \dots & A_{1,s+1} \\
\dots  & \dots & \dots \\
A_{r2}  & \dots & A_{r,s+1}
\end{pmatrix}
\\ & \qquad + \sum_{1\le j \le s}(-1)^{j}
f\begin{pmatrix}
A_{11}  & \dots & \left\{A_{1,j} A_{1,j+1} \right\}& \dots & A_{1,s+1} \\
\dots  & \dots& \dots& \dots & \dots \\
A_{r,1}  & \dots& \left\{A_{r,j}A_{1,j+1} \right\}& \dots & A_{r,s+1}
\end{pmatrix}
\\ & \qquad + (-1)^{s+1} f\begin{pmatrix}
A_{11}  & \dots & A_{1s} \\
\dots  & \dots & \dots \\
A_{r1}  & \dots & A_{rs} \end{pmatrix}.
\end{align*}
\end{itemize}
\vskip .5cm
A lengthy but straightforward computation shows that the following diagram
commutes
$$
\begin{CD}
D^{r+1, s} @>{d_H}>> D^{r+1, s+1}
\\
@AA{d_V}A @AA{d_V}A \\
D^{r, s} @>{d_H}>> D^{r, s+1}
\end{CD}
$$

Thus, using the usual ``sign trick'', we have constructed a double cochain complex

\begin{center}
\begin{equation}\label{totalcomplex}
D^{\cdot\cdot}(\T) = \qquad
\begin{CD}
\vdots \\
@AAA \\
D^{2, 0} @>{d_H}>> \hspace{18pt}\raisebox{1.0ex}{\vdots}\cdots \\
@AA{d_V}A @AA{-d_V}A \\
D^{1, 0} @>{d_H}>> D^{1, 1} @>{d_H}>>
\hspace{18pt}\raisebox{1.0ex}{\vdots}\cdots \\
@AA{d_V}A @AA{-d_V}A @AA{d_V}A \vspace{9pt}\\
D^{0, 0} @>{d_H}>> D^{0, 1} @>{d_H}>>
D^{0, 2} @>>> \cdots\ \ .
\end{CD}
\end{equation}
\end{center}

\vskip .5cm

If we remove the edges of this double complex we obtain another
complex by setting $\A^{r, s}(\T):= \A^{r,s} := D^{r+1,s+1}$ and defining
$$
d_v^{r,s} := (-1)^{s} d_V^{r+1,s+1} \quad \text{and} \quad d_h^{r,s} :=  d_H^{r+1,s+1},
$$
with the ``sign trick", we have $d_v\circ d_h + d_h\circ d_v =0$.
\begin{center}
\begin{equation}\label{complexdoublecohomology}
\xymatrix{
&\vdots&\\
&A^{2,0} \ar[u]^{d_v^{2,0}} \ar[r]_{d_h^{2,1}}& A^{2,1} \\
\A^{\cdot \cdot}(\T) := &A^{1,0} \ar[u]^{d_v^{1,0}} \ar[r]_{d_h^{1,0}} & A^{1,1}\ar[r]_{d_h^{1,1}} \ar[u]^{d_v^{1,2}}& A^{1,2}\\
&A^{0,0} \ar[u]^{d_v^{0,0}}\ar[r]_{d_h^{0,0}} & A^{0,1}\ar[u]^{d_v^{0,1}}
\ar[r]_{d_h^{0,1}} & A^{0,2} \ar[r]_{d_h^{0,2}}\ar[u]^{d_v^{0,2}}& \cdots
}
\end{equation}
\end{center}
\end{section}
\vskip 1cm

Finally, we denote by $E^{\cdot \cdot}(\T)$ the double complex consisting only of the edges of $D^{\cdot \cdot}(\T)$.

\begin{section}{Extensions of double groupoids by abelian group bundles}\label{seccion4}

\begin{definition}
Let $(\B; \Vc, \Hc; \Pc)$ and $(\F; \Vc, \Hc; \Pc)$ be
two double groupoids and let $\Pi: \B \to \F$ be a
morphism of double groupoids. The \emph{double kernel}
of $\Pi$ is the set
$$Ker(\Pi) = \{ B \in \B\;:\; \Pi(B) = \Theta_p,\;\text{for some }\; p \in \Pc \}$$
\end{definition}

\begin{lemma}
Let $(\B; \Vc, \Hc; \Pc)$ and $(\F; \Vc, \Hc; \Pc)$ be
two double groupoids and let $\Pi: \B \to \F$ be a
morphism of double groupoids. Then the \emph{double kernel}
of $\Pi$ is contained in the abelian group bundle associated
to $\B$ and is, therefore, an abelian group bundle itself.
\end{lemma}
\begin{proof}
Straightforward.
\end{proof}
\begin{definition}
Let $(\F; \Vc, \Hc; \Pc)$ be a double groupoid and
let $p: \Kc \to \Pc$ be any abelian group bundle.
We say that a double groupoid $(\B; \Vc, \Hc; \Pc)$ is
an extension of $\F$ by $\Kc$ if there is an epimorphism of
double groupoids $\Pi: \B \to \F$ such that
$Ker(\Pi) = \Kc$.

If $(\B; \Vc, \Hc; \Pc)$ is an extension of
$(\F; \Vc, \Hc; \Pc)$ by $p: \Kc \to \Pc$ then we write
$$1 \to \Kc \hookrightarrow \B \twoheadrightarrow \F \to 1$$

\end{definition}
As is usual, once the concept of extensions of some
algebraic structure is defined,  we will to introduce a series
of definition related to it. From now on we will work with double
groupoids over fixed lateral groupoids $\Vc$ and $\Hc$.
\begin{definition}
Let $\B_1$ and $\B_2$ be two double groupoid extensions of
the double groupoid $\F$ by an abelian group bundle $p: \Kc \to \Pc$.
We say that $\B_1$ and $\B_2$ are \emph{isomorphic} if there is a triple
$(\phi, \Phi, \Psi)$ where $\phi: \Kc \to \Kc$ is a morphism of
abelian group bundles and, $\Phi: \B_1 \to \B_2$ and $\Psi: \F \to \F$
are isomorphisms  of double groupoids such that the following diagram
is commutative
\begin{equation}\label{isomorphic extensions}
\xymatrix{
1 \ar[r] & \ar @{} [dr] |{\circlearrowleft} \Kc \ar@{^{(}->}[r]^\iota \ar[d]_\phi &
\ar @{} [dr] |{\circlearrowleft} \B_1 \ar@{>>}[r]^{\Pi_1} \ar[d]_\Phi
& \F \ar[r] \ar[d]_\Psi & 1 \\
1 \ar[r] & \Kc \ar@{^{(}->}[r]_\iota & \B_2 \ar@{>>}[r]_{\Pi_2} & \F \ar[r] & 1
}
\end{equation}

We say that the two extensions $\B_1$ and $\B_2$ are
\emph{equivalent} if in the diagram \eqref{isomorphic extensions}
the morphisms $\phi$ and $\Psi$ are identities. That is,
if the following diagram is commutative
\begin{equation}
\xymatrix{
& & \B_1 \ar@{}[ddr] |{\circlearrowleft} \ar@{}[ddl] |{\circlearrowright} \ar[dd]_\Phi
\ar@{>>}[dr]^{\Pi_1}&&\\
1 \ar[r] & \Kc \ar@{^{(}->}[dr]_\iota \ar@{^{(}->}[ur]^\iota & &
\F \ar[r] & 1\\
& & \B_2 \ar@{>>}[ur]_{\Pi_2}&&
}
\end{equation}

\end{definition}

\begin{definition}
Let $(\F; \Vc, \Hc; \Pc)$ be a double groupoid and
let $p: \Kc \to \Pc$ be any abelian group bundle.
The set of equivalence classes of extensions of
double groupoids of $\F$ by $\Kc$ will be
denoted by $\mathcal{O}pext(\F,\Kc)$.
\end{definition}
In the rest of the section we will to give a cohomological
description of $\mathcal{O}pext(\F,\Kc)$.

\subsection{The total complex associated to a double groupoid}

We will to recall the definition of the total complex
$Tot(A^{\cdot\cdot})$ associated to the double complex
$ A^{\cdot\cdot}$ defined in above section.
Let
$$
\Tot( A^{\cdot\cdot})^n = \bigoplus_{p+q=n}  A^{p,q}.
$$
Then
$$
d^n: \Tot( A^{\cdot\cdot})^n  \to \Tot( A^{\cdot\cdot})^{n+1} \quad
\text{given by}\quad d =  d_h +  d_v ,
$$
define a map such that $d \circ d =0$.
Then $(\Tot( A^{\cdot\cdot}),d)$ is a cochain complex.
\begin{definition}\label{discrete_double_groupoid_cohomology}
We define the {\em total $n$-cohomology of the double groupoid
$\T$ with coefficients in $\Kc$}, as the $n$-cohomology group
\begin{equation}
\Ho^n_{\Tot} (\T,\Kc) := \Ho^n(\Tot( \A^{\cdot\cdot})).
\end{equation}
of the total complex $\A^{\cdot \cdot}(\T,\Kc)$.
\end{definition}
\begin{lemma}
The set $H^0(\T, \Kc)$ of $0$-cocycles is the set of maps
$$
\left\lbrace \alpha: \F \to \Kc\; : \;\alpha\left\lbrace
\begin{matrix}
A_{11}\\A_{21}
\end{matrix}
\right\rbrace = \alpha (A_{21}) + l(A_{21})^{-1} \cdot \alpha(A_{11})
\;\text{and}
\;\alpha(\{ A_{11} A_{12} \}) = b(A_{11})\cdot \alpha(A_{12})
+\alpha(A_{11})
\right\rbrace
$$
\end{lemma}
\begin{proof}
Straightforward.
\end{proof}
\begin{remark}
If the actions of  $\Hc$ and $\Vc$ on $\Kc$ are trivial, then
the \emph{$0$-cohomology} of the double groupoid $\T$ is the set
of functions from $\F$ to $\Kc$ that are \textit{morphisms of
groupoids}	 with respect to the horizontal and vertical groupoid
structures on $\F$.
\end{remark}

\begin{proposition}
Let $p: \Kc \to \Pc$ be a $\T$ module. There is an exact sequence
\begin{equation}\label{kes-formulagral}
\begin{aligned}
0 &\to  \Ho^1(\Tot( D^{\cdot\cdot}))\to \Ho^1(\Hc,\Kc) \oplus \Ho^1(\Vc,\Kc) \to
\Ho^0_{\Tot} (\T,\Kc) \\
&\to  \Ho^2(\Tot( D^{\cdot\cdot}))\to \Ho^2(\Hc, \Kc) \oplus \Ho^2(\Vc, \Kc) \to
\Ho^1_{\Tot} (\T,\Kc)\\
&\to \Ho^3(\Tot( D^{\cdot\cdot}))\to \Ho^3(\Hc, \Kc) \oplus \Ho^3(\Vc, \Kc)  \to \cdots \end{aligned}
\end{equation}
\end{proposition}
\begin{proof}
Let $A^{(0)} = 0$, $A^{(1)} = 0$, $A^{(n)} = \Tot(A^{\cdot\cdot})^{n-2}$, for $n\ge 2$. We call $A^{(\cdot)}$ the cochain complex induced by the cochain complex $\Tot(A^{\cdot\cdot})$.  Then, we have the short exact sequence
$$
\begin{CD} 0@>>> A^{(n)} @>>> \Tot(D^{\cdot\cdot})^{n}
@>>> \Tot(E^{\cdot\cdot})^{n}@>>> 0\end{CD}$$
that induces a short exact sequence of cochain complexes
$$
\begin{CD} 0@>>> A^{(\cdot)} @>>> \Tot(D^{\cdot\cdot})
@>>> \Tot(E^{\cdot\cdot})@>>> 0\end{CD}.$$

Thus we have a long exact sequence
\begin{align*}
0 \to \Ho^0(A^{(\cdot)}) \to \Ho^0&(\Tot(D^{\cdot\cdot})) \to \Ho^0(\Tot(E^{\cdot\cdot})) \\
&\to \Ho^1(A^{(\cdot)}) \to \Ho^1(\Tot(D^{\cdot\cdot})) \to \Ho^1(\Tot(E^{\cdot\cdot})) \to \Ho^2(A^{(\cdot)}) \to \cdots.
\end{align*}
As $\Ho^0(A^{(\cdot)}) =  \Ho^1(A^{(\cdot)}) = 0$ and $ \Ho^{n}(\Tot(A^{\cdot\cdot})) = \Ho^{n+2}(A^{(\cdot)})$, we have
\begin{align*}
0 \to \Ho^1(\Tot(D^{\cdot\cdot})) \to \Ho^1(\Tot(E^{\cdot\cdot})) \to \Ho^0(\Tot(A^{\cdot\cdot})) \to  \Ho^2(\Tot( D^{\cdot\cdot})) \to \cdots
\end{align*}
and it is clear that
\begin{equation*}
\Ho^n(\Tot( E^{\cdot\cdot})) =  \Ho^n(\Hc, \Kc)
\oplus \Ho^n(\Vc, \Kc), \quad n> 0.
\end{equation*}
Thus, we get the result.
\end{proof}
\vskip .3cm
\vskip .3cm
\begin{subsection}{The total first cohomology group $\Ho^1_{\Tot} (\T,\Kc)$}
Here we show some properties of \emph{total $1$-cocycles} that are well known in the case of group cohomology. Through this subsection $p: \Kc \to \Pc$ denote an abelian group bundle and $\T = (\F; \Vc, \Hc;\Pc)$ is a double groupoid acting along $p$ on the left. We also denote by $\A^{\cdot \cdot}$ the double complex defined in \eqref{complexdoublecohomology}.

\begin{lemma}\label{1cociclos} Let $(\tau,\sigma) \in A^{0,1} \oplus A^{1,0}$ then $d^1(\tau,\sigma) = 0$ if and only if for all $F,G,H\in \B$,
\begin{center}
\begin{equation}\label{taucocyclo}
\tau(F\; G)+\tau(\{ F G \} \; H) =\tau(F\; \{ G H \}) +
\big( b(F) \cdot \tau(G\, H)\big),\;\text{for all}\; F,G,H\in \B\;
\text{such that}\; \quad F\vert G \vert H;
\end{equation}
\begin{equation}\label{sigmacocyclo}
\sigma \left( \begin{matrix} G \\ H \end{matrix} \right) +
\sigma\left(\begin{matrix} F \\ \left\{\begin{matrix} G \vspace{-4pt}\\
H \end{matrix} \right\} \end{matrix} \right)
= \big( l(H)^{-1} \cdot\sigma\left(\begin{matrix}F \\ G\end{matrix}\right) \big) +
\sigma \left( \begin{matrix} \left\{ \begin{matrix}
F \vspace{-4pt} \\ G \end{matrix} \right\} \\ H \end{matrix} \right),
\;\text{for all}\; F,G,H\in \B\;
\text{such that}\; \quad \begin{tabular}{p{0,4cm}}$F$ \\ \hline $G$\\
\hline $H$\end{tabular};
\end{equation}
\end{center}
and
\begin{center}
\begin{equation}\label{sigma-tau-compatibles}
\big(l(H)^{-1}\cdot \tau(F\,G)\big) +
\tau(H\,J) + \sigma\left(\begin{matrix}\left\{FG\right\} \\
\left\{HJ\right\}\end{matrix}\right) = \big(b(H)\cdot
\sigma\left(\begin{matrix}G\\J\end{matrix}\right)\big) +
\sigma\left(\begin{matrix}F\\H\end{matrix}\right) +
\tau\left(\left\{\begin{matrix} F \vspace{-4pt}\\ H
\end{matrix}\right\}\, \left\{\begin{matrix} G\vspace{-4pt} \\
J \end{matrix}\right\} \right),
\end{equation}
\end{center}
for all
\begin{tabular}
{p{0,4cm}|p{0,4cm}} $F$ & $G$ \\ \hline $H$ & $J$
\end{tabular}.
\end{lemma}
\begin{proof}
Let $(\sigma, \tau)$ be a total $1$-cocycle. Since
$d^1  : A^{1,0} \oplus A^{0,1} \to A^{2,0} \oplus A^{1,1} \oplus A^{0,2}$
is defined by  $d^1 = (d_v^{1,0}, d_h^{1,0} + d_v^{0,1}, d_h^{0,1})$,
and since
\begin{align*}
d_v^{1,0} := d_V^{2,1} &: D^{2,1} \to D^{3,1},\\
d_v^{0,1} := -d_V^{1,2} &: D^{1,2} \to D^{2,2},\\
d_h^{0,1} := d_H^{1,2} &: D^{1,2} \to D^{1,3},\\
d_h^{1,0} := d_H^{2,1} &: D^{2,1} \to D^{1,1}.
\end{align*}
Then $d^1(\tau,\sigma) = 0$ if and only if
$d_V^{2,1}(\sigma)  = 0$, $d_H^{2,1}(\sigma) - d_V^{1,2}(\tau) = 0$
and $d_H^{1,2}(\tau) =0$. These equations can be seen are equivalent
to \eqref{sigmacocyclo},\eqref{taucocyclo} and \eqref{sigma-tau-compatibles}.
\end{proof}
\begin{lemma}
Any total $1$-cocycle is cohomologous to a normalized total $1$-cocyle.
\end{lemma}
\begin{proof}
Let $(\sigma, \tau)$ be a total $1$-cocycle. We define
$\lambda,\;\delta:\B \to \Kc $ by
$\lambda(B) = \sigma \left( \begin{matrix} Id\; b(B)\\Id\; b(B)\end{matrix} \right)$
and
$\delta ( B ) = \tau ( Id\;l(B) \; Id\;l(B))$.
By mean of the equations \eqref{sigmacocyclo} and \eqref{taucocyclo},
it is easy to show that the pair $(\sigma', \tau')$, where
$\sigma'= \sigma - d_v^{1,0}\lambda$ and $\tau' = \tau - d_h^{0,1}\delta$,
is a normalized total $1$-cocycle.
\end{proof}
\begin{remark}
 In the proof of the above lemma we have really proved that any
 \emph{vertical and horizontal $2$-cocycles} are \emph{equivalent} to a
 normalized ones.
\end{remark}
\end{subsection}

\begin{subsection}{Extensions of double groupoids by abelian group bundles and the 1-cohomolgy group}

Recall that $\gamma: \F \to \Pc$ is the `left-bottom' vertex map, i.e. $\gamma(A) =
lb(A)$.

\begin{proposition}\cite[prop. 1.5]{AN3}\label{construct-double-groupoid}
Let $p: \Kc \to \Pc$ be an abelian group bundle, and let $\T = (\F; \Vc, \Hc;\Pc)$ be a double groupoid acting along $p$ on the left.
Let $(\sigma, \tau)$ be a normalized total $1$-cocycle of $\F$
with values in $\Kc$. Define on $\Kc \Times{p}{\gamma} \F$
the following maps
\begin{itemize}
\item Four maps $t,b,l,r$ on $\Kc\Times{p}{\gamma} \F$
      defined by those in $\F$: $t(K, F) = t(F)$ and so on.
\item Two composition laws, vertical and horizontal, in $\Kc\Times{p}{\gamma} \F$
    defined by
\begin{align}\label{prodhorreconstrbis}
\{(K,F)(L,G)\} &= \big(K +(b(F)\cdot L)+\tau(F\, G), \{FG\}\big),
\qquad \text{if }\; F\;\vert\; G, \\
\label{prodverreconstrbis}
\left\{\begin{matrix}(K,F) \\(L,G) \end{matrix}\right\} &=
\left( (l(G)^{-1}\cdot K)+ L+
\sigma \left(\begin{matrix}F \\ G \end{matrix}\right),
\left\{\begin{matrix}F \\ G\end{matrix}\right\}\right), \qquad \text{if }\; \dfrac FG\;.
\end{align}

\item Two identity maps $\id: \Vc\to\Kc\Times{p}{\gamma} \F$ and
$\id: \Hc\to\Kc\Times{p}{\gamma} \F$, given by
$\id g = (\Theta_{b (g)}, \id g)$ and $\id\; x = (\Theta_{l (x)}, \id\; x)$
if $g\in \Vc$ and $x\in \Hc$ respectively.

\item The inverse of $(K, F)$ with respect to the horizontal and
vertical products are given by
\begin{align}\label{inversa-hor}
(K,F)^h &= \left(b(F)^{-1}\cdot\left( -K - \tau(F\;, F^h)\right), F^h\right)
\quad \text{and}  \\
\label{inversa-ver} (K,F)^v &= \left(-\left(l(F)\cdot
K\right)-\sigma\left(\begin{matrix}F \\ F^v\end{matrix}\right), F^v\right),
\end{align}
respectively.
\end{itemize}
With these structural maps, the arrangement
$$\begin{matrix}
\Kc\Times{p}{\gamma} \F &\rightrightarrows &\Hc
\\\downdownarrows &&\downdownarrows \\ \Vc &\rightrightarrows &\Pc
\end{matrix}$$
is a double groupoid which will be denoted by $\Kc\; \sharp_{\sigma, \tau}\;\F$.
\end{proposition}
\begin{proof} It is easy to see that the  explicit hypothesis of the original formulation of this proposition are equivalent to ask for an action of $\F$ along $p$ and that $(\sigma, \tau)$ is a total $1$-cocycle (see Lemma \ref{1cociclos}).
  
The proof is a long but straightforward by checking each axiom
of the definition of double groupoid, see \cite[prop. 1.5]{AN3}
for further details.
\end{proof}
\begin{remark}
The proposition \ref{construct-double-groupoid} is, in fact, an if a only if result as one can easily check. That is, $\Kc\Times{p}{\gamma} \F$ is a double groupoid if and only if $d^1(\sigma, \tau) = 0$
\end{remark}

\begin{definition} Let $p: \Kc \to \Pc$ be an abelian group bundle, and let $\T = (\F; \Vc, \Hc;\Pc)$ be a double groupoid acting along $p$ on the left.
Let $(\sigma, \tau)$ be a normalized total $1$-cocycle of $\F$ with values in $\Kc$. The double groupoid structure defined on $\Kc\Times{p}{\gamma} \F$ as in proposition \ref{construct-double-groupoid} is called the  {\em double groupoid extensions of $\F$ by $(\sigma, \tau)$} and is denotated $\Kc\; \sharp_{\sigma, \tau}\;\F$ .
\end{definition}

Now, we are going to state the main result of
this section.

\begin{theorem}\label{corresp-cocycles-opext}
Let $(\F; \Vc, \Hc; \Pc)$ be a double groupoid and
let $p: \Kc \to \Pc$ be any abelian group bundle.
There is a bijection between $\mathcal{O}pext(\F,\Kc)$
and the cohomology group $\Ho^1_{\Tot}(\F, \Kc)$.
\end{theorem}
We divide the proof of this theorem in three steps. The first two ones are the contents of  lemma \ref{eq.ext.cohom.cocycles} and proposition \ref{any.ext.is.smash}.

\begin{lemma}\label{eq.ext.cohom.cocycles}
Let $(\F; \Vc, \Hc; \Pc)$ be a double groupoid,
let $p: \Kc \to \Pc$ be any abelian group bundle and
let $(\sigma, \tau)$ and $(\sigma', \tau')$ be normalized total $1$-cocycles of $\F$
with values in $\Kc$. If the double groupoid extensions
$\Kc\; \sharp_{\sigma, \tau}\;\F$ and $\Kc\; \sharp_{\sigma', \tau'}\;\F$
are equivalent, then $(\sigma, \tau)$ and $(\sigma', \tau')$
are in the same cohomology class.
\end{lemma}
\begin{proof}
Let $\Phi: \Kc\; \sharp_{\sigma, \tau}\;\F \to \Kc\; \sharp_{\sigma', \tau'}\;\F $
be an equivalence of extensions of double groupoids
\begin{equation}\label{eq.ext.cocycles}
\xymatrix{
& & \Kc\; \sharp_{\sigma, \tau}\;\F \ar@{}[ddr] |{\circlearrowleft} \ar@{}[ddl] |{\circlearrowright} \ar[dd]_\Phi
\ar@{>>}[dr]^{\Pi_1}&&\\
1 \ar[r] & \Kc \ar@{^{(}->}[dr]_\iota \ar@{^{(}->}[ur]^\iota & &
\F \ar[r] & 1\\
& & \Kc\; \sharp_{\sigma', \tau'}\;\F \ar@{>>}[ur]_{\Pi_2}&&
}
\end{equation}
and let $\lambda: \F \to \Kc$ be defined by $\lambda(F) = (\rho_1 \circ \Phi)(\theta, F)$,
where $\rho_1: \Kc\; \sharp_{\sigma', \tau'}\;\F \to \Kc$ denote the projection onto $\Kc$.

The following identities are easily obtained
from the operations on $\Kc\;\sharp_{\sigma, \tau}\;\F$
 and the commutativity of diagram \eqref{eq.ext.cocycles}
\begin{equation}\label{eq.prod.ext}
(K, F) = \left\{\begin{matrix} (\Theta, F) \\ (K, \Theta) \end{matrix}\right\},
\quad
\left\{ \begin{matrix} (\Theta, F) \\ (\Theta, G) \end{matrix} \right\}
=\left( \sigma\left(\begin{matrix} F \\ G \end{matrix}\right),
\left\{ \begin{matrix} F \\ G \end{matrix} \right\} \right)
\quad \text{and} \quad \Phi(K , F) = ((\rho_1 \circ \Phi)(K, F), F).
\end{equation}
From the above equations we obtain that
$$\sigma\left( \begin{matrix} F \\ G \end{matrix} \right)
= (\rho_1 \circ \Phi)(\Theta, G) - (\rho_1 \circ \Phi)\left(\Theta,
\left\{\begin{matrix} F \\ G \end{matrix}\right\}
\right) + l(G)^{-1} \cdot (\rho \circ \Phi)(\Theta, F)
+ \sigma'\left( \begin{matrix} F \\ G \end{matrix} \right)
\quad\text{for any} \quad F, G \in \F;
$$
that is $\sigma = \sigma' + d_V^{1,1}\lambda$. In the same way we
show that $\tau = \tau' + d_H^{1,1}\lambda$.
\end{proof}
\begin{proposition}\cite[Prop. 1.9]{AN3}\label{any.ext.is.smash}
Any double gropoid extension of $\F$ by $\Kc$ is 
equivalent to $\Kc \; \sharp_{\sigma, \tau} \F$ for
some total $1$-cocycle $(\sigma, \tau)$.
\end{proposition}
\begin{proof}
See \textit{loc. cit} for further details.
\end{proof}
Now we are going to provide a proof of Theorem \ref{corresp-cocycles-opext}.
\begin{proof}[Proof of theorem \ref{corresp-cocycles-opext}]
Let $(\sigma, \tau )$ and $(\sigma', \tau')$ be cohomologous
total $1$-cocycles. Then there is a map $\lambda: \F \to \Kc$
such that $(\sigma', \tau') = (\sigma, \tau) + d^0 \lambda$;
that is, $\sigma' = \sigma + d^{1,1}_V$ and $\tau'= \tau + d^{1,1}_H$.
Let $\B= \Kc\; \sharp_{\sigma, \tau}\;\F$ and by
$\B' = \Kc\; \sharp_{\sigma', \tau'}\;\F$ the extensions of $\F$ by
$\Kc$ determined in  \ref{construct-double-groupoid} by
$(\sigma, \tau)$ and $(\sigma', \tau')$, respectively.
Define the map
$$ \Phi: \B' \to \B \quad \text{such that} \quad (K, F) \mapsto (K + \lambda(F), F).$$
We state that $\Phi$ is an isomorphism of double groupoids.
In fact, it is clear that it is one to one and onto, then
we need to show that it preserve both structure of groupoids.

For all $(K, F), (L, G) \in \B'$ we have
\begin{align}
\Phi \left(
\begin{matrix}
(K, F) \\
(L, G)
\end{matrix}
\right)
&=\Phi \left((l(G)^{-1} \cdot K)+ L + \sigma'\left(\begin{matrix}F \\ G \end{matrix}\right),
\left\{ \begin{matrix}F \\ G \end{matrix}\right\}\right)\nonumber\\
&=\left( (l(G)^{-1} \cdot K)+ L +  \sigma'\left(\begin{matrix}F \\ G \end{matrix}\right) +
\lambda\left( \left\{ \begin{matrix}F \\ G \end{matrix}\right\}\right),
\left\{\begin{matrix}F \\ G \end{matrix}\right\}\right)\nonumber\\
&= \left( (l(G)^{-1} \cdot K) + L + (\sigma + d_V^{1,1})
\left( \begin{matrix} F \\ G \end{matrix}\right),
\left\{ \begin{matrix}F \\ G \end{matrix}\right\}  \right)\nonumber\\
&= \left( (l(G)^{-1} \cdot K) + L +
\sigma\left( \begin{matrix} F \\ G \end{matrix}\right)
+ \lambda(G) - \lambda \left( \left\{ \begin{matrix} F \\ G \end{matrix}\right\} \right)
+l(G)^{-1}\cdot \lambda(F) + \lambda \left(\left\{ \begin{matrix}
F \\G \end{matrix}\right\}\right),
\left\{ \begin{matrix}F \\ G \end{matrix}\right\}  \right)\nonumber\\
\label{expr.1}
&= \left( (l(G)^{-1} \cdot K) + L +
\sigma\left( \begin{matrix} F \\ G \end{matrix}\right)
+ \lambda(G) +l(G)^{-1}\cdot \lambda(F),
\left\{ \begin{matrix}F \\ G \end{matrix}\right\}  \right).
\end{align}
On the other side
\begin{align}\label{expr.2}
\left\{\begin{matrix}
\Phi(K, F) \\ \Phi(L, G)
\end{matrix}
\right\}
&= \left\{\begin{matrix}
(K + \lambda(F), F) \\ (L + \lambda(G), G)
\end{matrix}
\right\} = \left( l(G)^{-1} \cdot (K + \lambda(F)) + L + \lambda(G)
 + \sigma\left( \begin{matrix} F \\ G \end{matrix} \right),
 \left\{ \begin{matrix} F \\ G \end{matrix} \right\} \right);
\end{align}
and since the action of $\Vc$ distributes then the
expressions \eqref{expr.1} and \eqref{expr.2} coincides and
the map $\Phi$ preserves vertical composition. In the same
way we can show that $\Phi$ preserves horizontal composition,
and thus we have an equivalence of the extensions
$\B$ and $\B'$.
The above reasoning permit us to introduce a well defined
map
$$\Psi: H^1_{Tot}(\F, \Kc) \to \mathcal{O}pext(\F, \Kc),
\quad \text{such that}\quad [\sigma, \tau] \mapsto \Kc\; \sharp_{\sigma, \tau}\;\F,$$
where $[\sigma, \tau]$ denote the cohomology class
of a total $1$-cocycle $(\sigma, \tau)$; we will to show that
$\Psi$ is a bijection.

Clearly $\Psi$ is one to one by lemma \ref{eq.ext.cohom.cocycles}
and it is onto because of proposition \ref{any.ext.is.smash}. This
finish the proof.	
\end{proof}
\end{subsection}

%
%
%
%
%
%
%
%

\end{section}

\begin{section}{Bisimplicial spaces and double groupoids}

We are going to introduce some notation to be used
in the rest of the work. Let $\Delta$ be the simplicial category, that is,
the category whose objects are $[n] = \{ 0, 1, 2 ,\cdots, n \}$,
for  every $n \in \mathbb{N}$, and whose morphisms
are the order preserving maps.
The following proposition give us a combinatorial description of the simplicial category.

\begin{proposition}[Prop. VII.5.2 \cite{ml}]
The category $\Delta$, with objects all finite ordinals,
is generated by the arrows $\epsilon_i^n: [n-1] \to [n]$ and
$\eta_i^n: [n+1] \to [n]$, where $\epsilon_i^n$ and
$\eta_i^n$ are the unique increasing map that avoids $i$,
and the unique non-decreasing surjective map
such that $i$ is reached twice ($0 \leq i \leq n$), respectively.
These maps are subject to the following relations
\begin{itemize}
\item $\epsilon_i^{n-1}  \epsilon_j^n = \epsilon_{j-1}^{n-1} \epsilon_i^n$
if $i < j$,
\item $\eta_i^{n+1}\eta_j^n = \eta_{j+1}^{n+1} = \eta_{j+1}^{n+1} \eta_i^n$
if $i \leq j$,
\item $\epsilon_j^{n+1}\eta_j^n = \eta_{j-1}^{n-1}\epsilon_i^n$ if $i < j$,
\item $\epsilon_i^{n+1} \eta_j^n = \eta_j^{n-1} \epsilon_{i-1}^n$ if
$i > j + 1$,
\item $\epsilon_j^{n+1} \eta_j^n = \epsilon_{j+1}^{n+1} \eta_j^n = Id$.
\end{itemize}

\end{proposition}
\begin{proof}
See \cite[Prop. VII.5.2]{ml}.
\end{proof}
For a complete treatment
of the simplicial category and its properties see \cite{ml}
or \cite{gj} for a deepest one.

In the following diagram we show from left to right all the $\epsilon$  maps,
and from right to left all the $\eta$ maps of the above proposition,

\begin{equation}\label{simplicial-cat-diagram}
\xymatrix
{
[0] \ar@<4pt>[r] \ar@<-4pt>[r] & [1] \ar@<7pt>[r]\ar[r]\ar@<-7pt>[r] \ar[l] & 
[2] \ar@<10pt>[r]\ar@<4pt>[r]\ar@<-4pt>[r]\ar@<-10pt>[r] \ar@<4pt>[l] \ar@<-4pt>[l] & [3] \ar[l] \ar@<7pt>[l] \ar@<-7pt>[l]
\cdots.
}
\end{equation}

\subsection{Simplicial sets}
\begin{definition}
A simplicial set is a contravariant functor  $X: \Delta^{op} \to \mathcal{S}ets$,
where $\mathcal{S}ets$. In the same way we define \emph{simplicial topological spaces} (or just \emph{simplicial spaces})
and \emph{simplicial manifolds}, just changing $\mathcal{S}ets$ by $\mathcal{T}op$ or $\mathcal{M}an$,
the categories of \emph{topological spaces} or \emph{smooth manifolds}, respectively.
\end{definition}

\begin{remark}
Given a simplicial set $X$ we will to denote it by $X_\bullet = \{ X_n \}_{n \in \mathbb{N}}$.
We also write $\epsilon_i^n$ and $\eta_i^n$ for $\Delta(\epsilon_i^n)$ and $\Delta(\eta_i^n)$,
respectively, if no confusion arise.
\end{remark}

Any topological (Lie) groupoid $\xymatrix{ \mathcal{G} \ar@<-2pt>[r] \ar@<2pt>[r] & \Pc}$
canonically gives rise to a simplicial space (manifold)
as follows  \cite{tu}: Let
$$G_n = \{ (g_1, \ldots, g_n) \in \G^n\; | \; s(g_i) = t(g_{i+1}) \; \forall i \}$$
be the set of all composable $n$-tuples of arrows in $\G$
and define the \emph{face} and \emph{degeneracy maps} as follows
\begin{itemize}
\item $\tilde{\ep}_0^1(g) = r(g)$ and $\bar{\ep}_1^1(g) = s(g)$ for $n >1$;
\item $\tilde{\ep}_0^n(g_1, \ldots, g_n) = (g_2, \ldots, g_n)$ for $n >1$;
\item $\tilde{\ep}_n^n(g_1, \ldots, g_n) = (g_1, \ldots, g_{n-1})$ for $n >1$;
\item $\tilde{\ep}_i^n(g_1, \ldots, g_n) = (g_1, \ldots, g_ig_{i+1}, \ldots, g_n)$ for $1 \leq i \leq n-1$;
\item $\tilde{\eta}_0^0: \G_0 \to \G_1$ the unit map of the groupoid;
\item $\tilde{\eta}_0^n(g_1,\ldots, g_n) = (s(g_1), g_1, \ldots, g_n)$;
\item $\tilde{\eta}_i^n(g_1,\ldots, g_n) = (g_1, \ldots, g_i, t(g_i), g_{i+1}, \ldots, g_n )$
for $1 \leq i \leq n$.
\end{itemize}
We refer to loc. cit. to another way to see the simplicial structure of $\G_\bullet$.

\begin{remark}
We would like to note here that we are reading the arrows of 
a groupoid from left to right.
\end{remark}
\begin{definition}
Let $X$ and $Y$ be simplicial sets. A map of simplicial sets $f : X \to Y$ is
a natural transformation of contravariant set valued functors.
\end{definition}
We will to denote by $\Sc$ the resulting category of
simplicial sets, that is, $\Sc$ is the functor category $\mathcal{S}ets^{\Delta^{op}}$.

\subsection{Bisimplicial sets}

\begin{definition}
A bisimplicial set is a simplicial object in $\Sc$. That is,
a bisimplicial set $X$ is a functor
$$X : \Delta^{op} \times \Delta^{op} \to \mathcal{S}ets,$$
or equivalently, is a functor $X: \Delta^{op} \to 	\mathbb{S}$.
We will to denote a bisimplicial set $X$ by $X_{\bullet \bullet}$.
\end{definition}

In the following diagram we show, for the category $\Delta^2$, 
a bidimensional analogue to diagram \ref{simplicial-cat-diagram},
\begin{equation}
\xymatrix
{
 \vdots & \vdots  &\vdots  & \\
([2], [0]) \ar@<4pt>[r] \ar@<-4pt>[r]  \ar@<10pt>[u]\ar@<4pt>[u]\ar@<-4pt>[u]\ar@<-10pt>[u] \ar@<4pt>[d] \ar@<-4pt>[d] & ([2], [1]) \ar@<7pt>[r]\ar[r]\ar@<-7pt>[r] \ar[l] \ar@<10pt>[u]\ar@<4pt>[u]\ar@<-4pt>[u]\ar@<-10pt>[u] \ar@<4pt>[d] \ar@<-4pt>[d] & ([2], [2]) \ar@<10pt>[u]\ar@<4pt>[u]\ar@<-4pt>[u]\ar@<-10pt>[u] \ar@<4pt>[d] \ar@<-4pt>[d] \ar@<10pt>[r]\ar@<4pt>[r]\ar@<-4pt>[r]\ar@<-10pt>[r] \ar@<4pt>[l] \ar@<-4pt>[l] & \cdots \\
([1], [0]) \ar@<4pt>[r] \ar@<-4pt>[r] \ar@<7pt>[u]\ar[u]\ar@<-7pt>[u] \ar[d] & ([1], [1]) \ar@<7pt>[r]\ar[r]\ar@<-7pt>[r] \ar[l] \ar@<7pt>[u]\ar[u]\ar@<-7pt>[u] \ar[d]& ([1], [2]) \ar@<7pt>[u]\ar[u]\ar@<-7pt>[u] \ar[d] \ar@<10pt>[r]\ar@<4pt>[r]\ar@<-4pt>[r]\ar@<-10pt>[r] \ar@<4pt>[l] \ar@<-4pt>[l] & \cdots \\
([0], [0]) \ar@<4pt>[r] \ar@<-4pt>[r]  \ar@<4pt>[u] \ar@<-4pt>[u]  & ([0], [1]) \ar@<4pt>[u] \ar@<-4pt>[u] \ar@<7pt>[r]\ar[r]\ar@<-7pt>[r] \ar[l]  & ([0], [2]) \ar@<4pt>[u] \ar@<-4pt>[u]  \ar@<10pt>[r]\ar@<4pt>[r]\ar@<-4pt>[r]\ar@<-10pt>[r] \ar@<4pt>[l] \ar@<-4pt>[l] & \cdots.
}
\end{equation}

Where the depicted maps are as follows:
\begin{itemize}
\item the \emph{vertical face maps} $([m], [n]) \to ([m+1], [n])$ are $\epsilon^{m+1,n}_{i,v} = (\epsilon_i^{m+1}, Id)$,
\item the \emph{vertical degeneracy maps} $([m+1], [n]) \to ([m], [n])$ are $\eta^{m,n}_{i ,v}= (\eta_i^m, Id)$,
\item the \emph{horizontal face maps} $([m], [n]) \to ([m], [n+1])$ are $\epsilon_{i,h}^{m, n+1} = (Id, \epsilon^{n+1}_i)$, and
\item the \emph{horizontal degeneracy maps} $([m], [n+1]) \to ([m], [n])$ are $\eta^{m, n}_{i,h} = (Id, \eta^{n}_i)$.
\end{itemize}

Thus, a bisimplicial set $\dbs{X}$ can be depict as an array
\begin{equation}
\xymatrix
{
 \vdots \ar@<10pt>[d]\ar@<4pt>[d]\ar@<-4pt>[d]\ar@<-10pt>[d]  & \vdots \ar@<10pt>[d]\ar@<4pt>[d]\ar@<-4pt>[d]\ar@<-10pt>[d] &\vdots  \ar@<10pt>[d]\ar@<4pt>[d]\ar@<-4pt>[d]\ar@<-10pt>[d] & \\
X_{2, 0} \ar[r] \ar@<7pt>[d]\ar[d]\ar@<-7pt>[d] & X_{2, 1} \ar@<4pt>[r] \ar@<-4pt>[r] \ar@<4pt>[l] \ar@<-4pt>[l] \ar@<7pt>[d]\ar[d]\ar@<-7pt>[d] & X_{2, 2}  \ar@<7pt>[l]\ar[l]\ar@<-7pt>[l] \ar@<7pt>[d]\ar[d]\ar@<-7pt>[d] & \cdots \ar@<10pt>[l]\ar@<4pt>[l]\ar@<-4pt>[l]\ar@<-10pt>[l] \\
X_{1, 0} \ar[r] \ar@<-4pt>[d] \ar@<4pt>[d] \ar@<4pt>[u] \ar@<-4pt>[u] & X_{1, 1} \ar@<4pt>[r] \ar@<-4pt>[r] \ar@<4pt>[l] \ar@<-4pt>[l] \ar@<-4pt>[d] \ar@<4pt>[d] \ar@<4pt>[u] \ar@<-4pt>[u] & X_{1, 2} \ar@<7pt>[l]\ar[l]\ar@<-7pt>[l] \ar@<-4pt>[d] \ar@<4pt>[d] \ar@<4pt>[u] \ar@<-4pt>[u] & \cdots \ar@<10pt>[l]\ar@<4pt>[l]\ar@<-4pt>[l]\ar@<-10pt>[l] \\
X_{0, 0} \ar[r] \ar[u] & X_{0, 1}  \ar@<4pt>[r] \ar@<-4pt>[r] \ar@<4pt>[l] \ar@<-4pt>[l] \ar[u] & X_{0, 2}  \ar@<7pt>[l]\ar[l]\ar@<-7pt>[l] \ar[u]  & \cdots \ar@<10pt>[l]\ar@<4pt>[l]\ar@<-4pt>[l]\ar@<-10pt>[l] .
}
\end{equation}

With \emph{face and degeneracy maps} denoted by $\tilde{\epsilon}^{m+1,n}_{i,v}, \tilde{\eta}^{m,n}_{i ,v},
\tilde{\epsilon}_{i,h}^{m, n+1}$ and $\tilde{\eta}^{m, n}_{i,h}$.

\begin{definition}
Let $X$ be a bisimplicial set. The \emph{diagonal simplicial set $d(X)$ associated
to $X$} is the simplicial set defined as $d(X)_n = X_{n,n}$. It
also can be viewed as the composition functor
$$ \Delta^{op} \overset{\Delta}\to \Delta^{op} \times \Delta^{op}
 \overset{X} \to \mathcal{S}ets,$$
 where $\Delta$ is the diagonal functor.
\end{definition}

\subsection{A geometric construction of the bisimplicial set 
associated to a double groupoid}
\label{bisimplicial set of a double groupoid}
It is well known that to any category (in particular to any groupoid)
we can associate a simplicial set, called its categorical nerve.
In the same way, with any double category or double groupoid, 
we can associate a \emph{bisimplicial set} \cite[Sect. 3.5]{AN2}. In the case
of double groupoids, the existence of inverses for both operations permit us 
to construct the \text{nerve} of the double groupoid in 
a \emph{geometrical} way just from the core action on itself. 

We introduce some notation to construct the nerve of double groupoid. Let $(\F; \Vc, \Hc; \Pc)$ be a double groupoid, then if $x,x'\in \Hc$, $x\vert x'$ denotes that $r(x) =l(x')$. If $g,g' \in \Vc$, $g\vert g'$ denotes that $b(g) = t(g')$.  In analogous way, if $A,A' \in \F$, $A\vert A'$ denotes that $r(A) = l(A')$ and \,
$\begin{matrix}
A \\ \hline
A'
\end{matrix}$
\, denotes that $b(A) = t(A')$.

In section \ref{discrete_cohomology_section} we build a bigraded family of sets 
$\{ \mathcal{F}^{(m,n)} \}_{(m,n) \in \mathbb{N} \times \mathbb{N}}$ that allowed us to define the cohomology of discrete double groupoids. This family of sets can be endowed with a structure of bisimplicial set by defining the face and degeneracy maps as follows:
\begin{itemize}
\item For any $n \in \mathbb{N}$, let $\widetilde{\epsilon^{1,n}_{0,v}}[F_{1j}] = [b(F_{1j})]$;
\item For any $m >1$ and $n \in \mathbb{N}$, let $\widetilde{\epsilon^{m+1,n}_{0,v}}[F]$ be the array obtained from $F$ deleting the first row;
\item For any $m >1$ and $n \in \mathbb{N}$, let $\widetilde{\epsilon^{m+1,n}_{m+1,v}}[F]$ be the array obtained from $F$ deleting the last row;
\item For any $m >1$, $n \in \mathbb{N}$ and $1 \leq i \leq m$, let $\widetilde{\epsilon^{m+1,n}_{i,v}}[F]$ be the array obtained from $F$ composing the elements of rows $i$ and $i + 1$;
\item the \emph{vertical degeneracy maps} $([m+1], [n]) \to ([m], [n])$ are $\eta^{m,n}_{i ,v}= (\eta_i^m, Id)$,
\item the \emph{horizontal degeneracy maps} $([m], [n+1]) \to ([m], [n])$ are $\eta^{m, n}_{i,h} = (Id, \eta^{n}_i)$.
\end{itemize}

This bisimplicial set can be constructed in a more geometric way as a \emph{homogeneous space} of the core groupoid, in fact, 
%
%
%
%
%
%
%
let $\E(\F)$ be the 
core groupoid of $\F$. Let us denote by $\F^{m \times n}$ the set of $m \times n$ matrices 
with entries in $\F$. Let
\begin{equation}\label{pre-double-complex}
\F_{(m,n)} = \{ F = [F_{i,j}]_{0 \leq i \leq m , \; 0 \leq j \leq n} \in \F^{m \times n}\;|\; 
\l(F_{i.j}) = \l(F_{i, j+1}) \;\text{and }\; b(F_{i, j}) = b(F_{i+1, j})\},
\end{equation}
be the set of matrices of size $m \times n$ of boxes in $\F$ such that 
the boxes in a fixed row has the same left side and the boxes in 
a fixed column has de same bottom side.
We can note that the entries of a matrix in this set 
has the same left-bottom corner.

In \cite[Prop. 1.1]{AN3} the authors define the map 
$\gamma: \F \to \Pc$ as the ``left-bottom'' vertex 
$\gamma(B) = lb(B)$, and an action of the core groupoid $\E(\F)$ 
over the set of boxes $\F$ given be 
\begin{equation}\label{actioncore} E \rightharpoondown A : =
\left\{\begin{matrix}\iddv l(A)& A \vspace{-4pt}\\ E &\iddv
b(A)\end{matrix} \right\}, \quad A\in \F, E \in \E.
\end{equation}
We can extend this definitions to the sets of matrices of boxes $\F_{(m,n)}$, 
let us to define, for every $m,n \in \mathbb{N}$, the maps 
$\gamma_{(m,n)}: \F_{(m,n)} \to \Pc$ by 
$\gamma_{(m,n)}([F_{i.j}]) = bl(F_{m,1})$ and
let $\rightharpoondown : \E(\F) \pfib{e_{\E}}{\gamma_{(m,n)}}\F_{(m,n)}$ be
the action given by
\begin{center} 
$E \rightharpoondown [F_{i,j}] = [E \rightharpoondown F_{i,j}]$ for any $E \in \E(\F)$
and $[F_{i,j}] \in \F_{(m,n)}$ such that $\gamma([F_{i,j}]) = e_{\E(\F)}(E)$.
\end{center}

Denote by $\widetilde{\F}^{m,n}$ the set of orbits $\F_{m,n} / \E(\F)$ and 
by $\langle F_{i,j} \rangle$ the equivalence class  
in $\widetilde{\F}^{m,n}$ of a matrix of boxes $[F_{i,j}]$.

From now on we are going to denote the maps $\gamma_{(m,n)}$
bye the same letter $\gamma$.

\begin{proposition}\label{double_skeleton_core_groupoid}
Let $(\F; \Vc, \Hc; \Pc)$ be a double groupoid, let $\E(\F)$ be the 
core groupoid of $\F$ and let $\widetilde{\F}^{m,n}$ be as defined above. Then 
the map $\Phi: \F^{(m,n)} \to \widetilde{\F}^{m,n}$ given by $[F_{i,j}] \mapsto \langle \overline{F}_{k,l} \rangle$ ,
where 
$$ \overline{F}_{k,l} =
\begin{cases}
\left\{
\begin{matrix}
F_{k, 1} & F_{k,2}& \cdots & F_{k,l}\\
\vdots & \vdots & & \vdots \\
F_{m,1} & F_{m,2} & \cdots & F_{m,l}\\
\end{matrix}
\right\} & \text{if} \;\; l\neq 0 \;\text{and}\; k \neq m,\\
\quad\quad\quad \iddv & \text{if}\;\; k = m \;\text{and}\; l \neq 0,\\
\quad\quad\quad \iddv & \text{if}\;\; l = 0 \;\text{and}\; k \neq m,\\
\quad\quad\quad \Theta & \text{if}\;\; k = m\;\text{and} \;l =  0,
\end{cases}
$$
is bijective, with inverse $\Psi : \widetilde{\F}^{m,n} \to \F^{(m,n)}$, given by
$$
\Psi(\langle F_{i,j} \rangle) = 
\left[
\begin{matrix}
\left\{
\begin{matrix}
F_{0,0}^h & F_{0,1}\\
F_{1,0}^{hv} & F_{1,1}^v 
\end{matrix}
\right\}
& \cdots &
\left\{
\begin{matrix}
F_{0,n-1}^h & F_{0,n}\\
F_{1,n-1}^{hv} & F_{1,n}^v 
\end{matrix}
\right\} \\
\vdots & & \vdots \\
\left\{
\begin{matrix}
F_{m-1,0}^h & F_{m-1,1}\\
F_{m,0}^{hv} & F_{m,1}^v 
\end{matrix}
\right\}
& \cdots &
\left\{
\begin{matrix}
F_{m-1,n-1}^h & F_{m-1,n}\\
F_{m,n-1}^{hv} & F_{m,n}^v 
\end{matrix}
\right\}
\end{matrix}
\right]
$$
\end{proposition}
\begin{proof}
Fist we need to check that the map $\Psi$ is well defined.  In fact, 
if $\langle F_{i,j} \rangle = \langle L_{i,j} \rangle $ in $\widetilde{\F}^{m,n}$,
then there is $E \in \E(\F)$ such that 
$L_{i,j} =  E \rightharpoondown F_{i,j}$, for every $1\leq i\leq m$ and $1 \leq j \leq n$.
Then for $i=0, \ldots, m-1$ and $j= 0, \ldots, n-1$
we have
\begin{align*}
\left\{
\begin{matrix}
L_{i,j}^h & L_{i,j+1} \\
L_{i+1,j}^{hv} & L_{i+1,j+1}^v
\end{matrix}
\right\}
& = \left\{
\begin{matrix}
(E \rightharpoondown F_{i,j} )^{h} & (E \rightharpoondown F_{i,j+1}) \\
(E \rightharpoondown F_{i+1,j})^{hv} & (E \rightharpoondown F_{i+1,j+1})^{v}
\end{matrix}
\right\} \\
&= \left\{
\begin{matrix}
\left\{
\begin{matrix}
\iddv & F_{ij} \\
E & \iddv
\end{matrix}
\right\}^h
&
\left\{
\begin{matrix}
\iddv & F_{i,j+1}\\
E & \iddv
\end{matrix}
\right\}
\\
\left\{
\begin{matrix}
\iddv & F_{i+1,j}\\
E & \iddv
\end{matrix}
\right\}^{hv} &
\left\{
\begin{matrix}
\iddv & F_{i+1,j+1}\\
E & \iddv
\end{matrix}
\right\}^{v}
\end{matrix}
\right\} \\
& =
\left\{
\begin{matrix}
\left\{
\begin{matrix}
F_{ij}^h & \iddv^h \\
\iddv^h & E^h
\end{matrix}
\right\}
&
\left\{
\begin{matrix}
\iddv & F_{i,j+1} \\
E & \iddv
\end{matrix}
\right\} \\
\left\{
\begin{matrix}
\iddv^{hv} & E^{hv} \\
F_{i+1,j}^{hv} & \iddv^{hv}
\end{matrix}
\right\}
&
\left\{
\begin{matrix}
E^v & \iddv^v \\
\iddv^v & F_{i+1,j+1}^v
\end{matrix}
\right\}
\end{matrix}
\right\}\\
&=
\left\{
\begin{matrix}
F_{ij}^h &
\left\{
\begin{matrix}
\iddv^h & \iddv
\end{matrix}
\right\} &
F_{i,j+1} \\
\left\{
\begin{matrix}
 \iddv^h \\
\iddv^{hv}
\end{matrix}
\right\} &
\left\{
\begin{matrix}
E^h & E \\
E^{hv} & E^v
\end{matrix}
\right\}
& 
\left\{
\begin{matrix}
\iddv \\
\iddv^v
\end{matrix}
\right\} \\
F_{i+1,j}^{hv} & 
\left\{
\begin{matrix}
\iddv^{hv} & \iddv^v
\end{matrix}
\right\}
&
F_{i+1,j+1}^v
\end{matrix}
\right\}\\
&=
\left\{
\begin{matrix}
F_{ij}^h & F_{i,j+1} \\
F_{i+1,j}^{hv} & F_{i+1,j+1}^v
\end{matrix}
\right\}.
\end{align*}
The above shows that $\Psi$ es well defined.

By a direct computation we can show that $\Phi$ and $\Psi$ are inverse each other.
\end{proof}
\begin{remark}
With the identification introduced in the above result,
it is clear that the construction carried out in \ref{bsp-set-db-gpd}
associates a bisimplicial set to every double groupoid. In fact, if we have a morphism $f: ([m], [n]) \to ([k], [l])$ in $\Delta^2$ we can associate to it the map 
$$\tilde{f}: \F^{(k,l)} \to \F^{(m,n)} \quad \text{given by} \quad
\tilde{f}(\langle A_{i,j}\rangle_{0 \leq i \leq k , \; 0 \leq j \leq l} ) = \langle A_{f(i,j)} \rangle_{1 \leq i \leq m , \; 1 \leq j \leq n}.$$
It is no difficult to show that this maps define a contravariant functor 
from $\Delta^2$ to $\mathcal{S}ets$.
\end{remark}
%
%
%
%
%
%

\end{section}

\section{Bisimplicial cohomology of topological double groupoids}

The classification results obtained here (theorem \ref{corresp-cocycles-opext} and proposition \ref{any.ext.is.smash}) where proved for discrete double groupids. The next most natural step is to study to what extend we can carry out such decomposition of double groupoids in the topological and/or differentiable setting.

The first one attempt is to study continuous (differentiable) cohomology requiring that all maps that define the double groupoid cohomology be continuous (smooth) but, for the discrete case, the proof of proposition \ref{any.ext.is.smash} \cite[Prop. 1.9]{AN3}, that indicates that any double groupoid extension can be obtained as the \textit{smash product} of a slim double groupoid by an abelian group bundle, depends strongly in the existence of a section of the function that \textit{maps} a double groupoid onto its frame. In the continuous or  differentiable setting we cannot guarantee any more the existence of a global section of such map.

If the \textit{frame} of a topological (Lie) double groupoid is a quotient space (or smooth manifold) and the \textit{frame map} is an open map (surjective submersion, respectively) then we can assure the existence of local sections and then, we can try to localize the decomposition process of the double groupoid to the open sets where the local sections exists.
In this section we develop a \v{C}ech double groupoid cohomology that will permit us to classify the extensions of topological double groupoids by topological abelian group bundles as in the discrete case.

\subsection{\v{C}ech double groupoid}

\begin{definition}\label{cech_groupoid}
Let $\xymatrix{ \mathcal{G} \ar@<-2pt>[r] \ar@<2pt>[r] & \Pc}$
be a topological groupoid. Let $\{ U_i \}_{i \in I}$
be an open cover of $\Pc$, define the \textit{cover groupoid},
\textit{\v{C}ech groupoid} or the \textit{localization groupoid}
\begin{equation}
\mathcal{G}[U] = \{ (i, g, j) \in I \times \mathcal{G} \times I\;:\;
g \in r^{-1}(U_i) \cap s^{-1}(U_j)) \},
\end{equation}
with unity spaces $\mathcal{P}[U]=\{ (i, x) \in I \times \mathcal{P}\;:\; x \in U_i \}$,
source and target maps $s(i, g, j) = (i, s(g))$ and $r(i, g, j) = (j, r(g))$
and product $(i, g, j)(j, h, k)=(i, gh, k)$, when $r(g) = s(h)$.
\end{definition}
A relevant fact of the \textit{\v{C}ech groupoid}
is that the canonical map $\mathcal{G}[U] \to \mathcal{G}$
is a \emph{Morita equivalence} of groupoids.
In the case of double groupoids, given an open covering of 
the total base, we can construct a \emph{localization}
double groupoid in a similar fashion, although the notion of \emph{Morita equivalence} is more subtle and we don't study it here.
\begin{definition}
Let $(\B; \Vc, \Hc; \Pc)$ be a topological double groupoid and
let us consider $\mathcal{U} = \{ \mathcal{U}_i \}_{i \in I}$ be an open
cover of $\Pc$, we define de \textit{double cover groupoid of $(\B, \mathcal{U})$},
\textit{\v{C}ech double groupoid associated to $(\B, \mathcal{U})$}  or the
\textit{localization double groupoid associated to $(\B, \mathcal{U})$} as follows
\begin{itemize}
\item Let $\B[\mathcal{U}]$ be the set
\begin{equation}
\left\{ \left( \begin{matrix}
i & & j\\
& B & \\
l & & k
\end{matrix} \right) \;:\; B \in \B, \;
t(B) \in l^{-1}(U_i) \cap r^{-1}(U_j)\;
\text{and}\;
b(B) \in l^{-1}(U_l) \cap r^{-1}(U_k)
\right\}
\end{equation}
\item Let $\tau,\;\beta,\;\lambda$ and $\rho$
be the maps defined by
\begin{align}
&\tau,\beta: \B[\mathcal{U}] \to \Hc[U]\quad \text{such that}\quad
\tau\left( \begin{matrix}
i & & j\\
& B &\\
l & & k
\end{matrix} \right) = (i,t(B),j)
\quad \text{and}\quad
\beta\left( \begin{matrix}
i & & j\\
& B &\\
l & & k
\end{matrix} \right)
= (l, b(B), k)
\\
&\lambda,\rho: \B[\mathcal{U}] \to \Vc[U]\quad \text{such that}\quad
\lambda\left( \begin{matrix}
i & & j\\
& B &\\
l & & k
\end{matrix} \right) = (i,l(B),l)
\quad \text{and}\quad
\rho \left( \begin{matrix}
i & & j\\
& B &\\
l & & k
\end{matrix} \right)
= (j, r(B), k)
\end{align}
\item We define a horizontal and vertical composition laws
by the following rules
$$
\begin{matrix}
\left( \begin{matrix}
i & & j\\
& A &\\
l & & k
\end{matrix} \right)
\\
\left( \begin{matrix}
i' & & j'\\
& B &\\
l' & & k'
\end{matrix} \right)
\end{matrix}
=
\left( \begin{matrix}
i & & j\\
&
\left\{\begin{matrix}
A\\
B
\end{matrix}\right\}
&\\
l' & & k'
\end{matrix} \right)
\quad
\text{if}
\quad
(l, b(A), k) = (i', t(B), j')
$$
and
$$
\left( \begin{matrix}
i & & j\\
& A &\\
l & & k
\end{matrix} \right)
\left( \begin{matrix}
i' & & j'\\
& B &\\
l' & & k'
\end{matrix} \right)
=
\left( \begin{matrix}
i & & j\\
& \{ A \; B \} &\\
l & & k
\end{matrix} \right)
\quad
\text{if}
\quad
(j, r(A), k) = (i', l(B), l')
$$
\item The identity maps of the horizontal and vertical
groupoid structure of $\B[U]$ are defined as follows
\begin{align}
&Id_h: \Vc[U] \to \B[U] \quad \text{such that} \quad Id_h(j, f, k)
=
\left( \begin{matrix}
j & & j\\
& Id_h f &\\
k & & k
\end{matrix} \right)
\\
&Id_v:\Hc[U] \to \B[U] \quad \text{such that} \quad
Id_v (i, x, j)
=
\left( \begin{matrix}
i & & j\\
& Id_v x &\\
i & & j
\end{matrix} \right)
\end{align}
\item The horizontal and vertical inverse of a box in $\B[U]$
are defined respectively by
\begin{align}
&(\;-\;)^h:\B[U] \to \B[U] \quad
\text{such that} \quad
\left( \begin{matrix}
i & & j\\
& B & \\
l & & k
\end{matrix} \right)^h
=
\left( \begin{matrix}
j & & i\\
& B^h & \\
k & & l
\end{matrix} \right)\\
&(\;-\;)^v:\B[U] \to \B[U]
\quad
\text{such that}
\quad
\left( \begin{matrix}
i & & j\\
& B & \\
l & & k
\end{matrix} \right)^v
=
\left( \begin{matrix}
l & & k\\
& B^v & \\
i & & j
\end{matrix} \right)
\end{align}
\end{itemize}
With all the above maps and operations it is easily to check
that
$$
\xymatrix{
\B[U] \ar@<-2pt>[r]_{\lambda} \ar@<2pt>[r]^{\rho}
\ar@<-2pt>[d]_{\tau} \ar@<2pt>[d]^{\beta}
&
\Vc[U] \ar@<-2pt>[d]  \ar@<2pt>[d]
\\
\Hc[U] \ar@<-2pt>[r] \ar@<2pt>[r]&
\Pc[U]
}
$$
is a double groupoid.
\end{definition}
\begin{remark}
Roughly speaking, if we denote by
$$\B_{(i, j, k, l)} = t^{-1}(l^{-1}(U_i) \cap r^{-1}(U_j))\;
\bigcap \; b^{-1}(r^{-1}(U_k) \cap l^{-1}(U_l)),$$
then the \v{C}ech groupoid associated with $(\B, \mathcal{U})$
is the \textit{disjoint union double groupoid}
\begin{equation}
\xymatrix{
\bigsqcup_{(i, j, k, l) \in I^4}\B_{(i, j, k, l)}
\ar@<-2pt>[r] \ar@<2pt>[r] \ar@<-2pt>[d] \ar@<2pt>[d] &
\bigsqcup_{(i, j) \in I^2} \Vc_{(i, j)} \ar@<-2pt>[d] \ar@<2pt>[d]\\
\bigsqcup_{(l,k) \in I^2} \Hc_{(l,k)} \ar@<-2pt>[r]\ar@<2pt>[r] & \bigsqcup_{i \in I} \Pc_{i}.
}
\end{equation}
\end{remark}

\subsection{Sheaves and coverings on bisimplicial spaces}

\begin{definition}\cite[6.4.2]{D}\label{bisim-sheaf}
A sheaf $\ubs{\mathcal{S}}$ on a bisimplicial space $\dbs{M}$
is a collection $\{ \mathcal{S}^{m,n} \}_{(m,n) \in \mathbb{N}^2}$ such that
\begin{enumerate}
\item $\mathcal{S}^{m,n}$ is a sheaf on $M_{m,n}\;$;
\item for all morphisms $f \in Hom_{\Delta^2}((k,l), (m,n))$ we
are given $\tilde{f}$-morphisms
$$\tilde{f}^{\ast}:\mathcal{S}^{k,l} \to \mathcal{S}^{m,n},$$
\end{enumerate}
such that $\widetilde{f \circ g}^{\ast} = \tilde{f}^{\ast} \circ \tilde{g}^{\ast}$, when it is defined.
\end{definition}

The next definition introduces the open coverings that 
behaves well for the study of bisimplicial sheaves.
\begin{definition}\label{open_cover}
An open cover of a bisimplicial space $X$ is family
$\mathcal{U}_{\bullet \bullet} = \{ \mathcal{U}_{(m,n)} \}_{m, n \in \mathbb{N}}$				
such that $\mathcal{U}_{(m, n)} = \{ U^{m,n}_i \}_{i \in I_{(m,n)}}$
																																is an open cover of the space $X_{(m,n)}$.
																																The cover is said to be bisimplicial if $I_{\bullet \bullet} = \{ I_{(m,n)} \}$
is a bisimplical set such that for all $f \in Hom_{\Delta^2}((k,l),(m,n))$
and for all $i \in I_{(m,n)}$ we have that $\tilde{f}(U^{(m,n)}_i) \subseteq U^{(k,l)}_{\tilde{f}(i)}$.
\end{definition}

It is obvious that a randomly chosen open cover of $\dbs{M}$ could be far away to be bisimplicial. Nevertheless, as the following lemma shows, given an open cover of a bisimplicial space we can form a bisimplicial cover $bs(\mathcal{U}_{\bullet\bullet})$ that
\emph{refines} the original one.
\begin{lemma}\label{bscover}
To any open cover $\mathcal{U}^{\bullet \bullet}$ of a bisimplicial set, 
there is a \emph{naturally} associated bisimplicial open cover 
$bs(\mathcal{U}^{\bullet \bullet})$.
\end{lemma}
\begin{proof}
Let us consider $\mathbb{N}^2$ ordered by the lexicographic order.	
Let $\Pc_{m,n}^{k,l}: = \hom_{\Delta'^2}(([k],[l]),([m],[n]))$ and put 
$\Pc_{m,n} = \bigcup_{(k,l) \leq (m,n)} \Pc_{m,n}^{k,l}$. As in \cite[Sec. 4.1]{tu},
the set $\Pc_{m,n}$ can be identified with the set of pairs of 
nonempty subsets of $[m] \times [n]$, then the cardinality of 
$\Pc_{m,n}$ is $(2^{m+1}-1)(2^{n+1}-1)$.
Let 
$$\Lambda =
  \left\{\lambda: \Pc_{m,n} \to \bigcup_{k,l} I_{k, l} \; : \; 
\lambda(\Pc_{m,n}^{k,l}) \subseteq I_{k,l} \; ,\; \forall \; (k,l) \leq (m,n)  \right\},
$$
and for $\lambda \in \Lambda_{m,n}$ define 
$$V_{\lambda}^{m,n} = \bigcap_{(k,l) \leq (m,n)} \bigcap_{f \in \Pc_{(m,n)}^{(k,l)}}
\tilde{f}^{-1}(U_{\lambda(f)}^{(k,l)}).$$
It is clear that $(V_{\lambda}^{m,n})_{\lambda \in \Lambda_{(m,n)}}$
is an open cover of $M_{m,n}$. In fact, let $x \in M_{m,n}$ and let 
$f \in \Pc_{m,n}$. Since $\tilde{f}: M_{m,n} \to M_{k,l}$ and $\{ U_{i}^{k,l} \}_{i \in I_{k,l}}$
is a covering, then $\tilde{f}(x) \in U_i^{k,l}$ for some (non unique) $i \in I_{k,l}$.
If we name this $i$ as $\lambda(f)$, then we have a function 
$\lambda$ in $\Lambda_{m,n}$ and , by the same definition,
$x \in \Lambda_{m,n}$.
Now we are going to show that $V^{\bullet \bullet} = \{ V_{m,n} \}$
is an open bisimplicial covering of $M_{\bullet \bullet}$.
To do this we need to define a bisimplicial structure on 
$\Lambda_{\bullet \bullet}$.
Let $g \in Hom_{\Delta'^{2}}(([m],[n]),([m'],[n']))$, then we have
a continuous map $\tilde{g}: M_{m',n'} \to M_{m,n}$ and set map 
$\tilde{g}:I_{m',n'} \to I_{m,n}$ (the use of the same notation 
will be clear from the context).
Let $\tilde{g}: \Lambda_{m', n'} \to \Lambda_{m,n}$ defined as follows; 
given $\lambda' \in \Lambda_{m', n'}$
take $\tilde{g}(\lambda) \in \Lambda_{m,n}$ as the map defined by
$\tilde{g}(\lambda')(f) = \lambda'(g \circ f)$.

\noindent Let $x \in V_{\lambda'}^{m',n'}$, $(k,l) \leq (m,n)$
and $f \in \Pc_{m,n}^{k,l}$.
Since $g \circ f : ([k],[l]) \to ([m'],[n'])$
then $\widetilde{(g \circ f)}(x) \in U_{\lambda'(g \circ f)}^{k,l}$,
that is, $\tilde{f}(\tilde{g}(x)) \in U_{\tilde{g}(\lambda')(f)}^{k,l}$.
This mean that $\tilde{g}(x) \in \tilde{f}^{-1}(U^{k,l}_{\tilde{g}(\lambda')(f)})$
and therefore $\tilde{g}(x) \in U_{\tilde{g}(\lambda')}^{m,n}$.
With the above we have proved that $\tilde{g}(U^{m',n'}_{\lambda'}) \subseteq U_{\tilde{g}(\lambda')}^{m,n}$ 
or $\tilde{g}(\lambda') \in \Lambda_{m,n}$.
\end{proof}
\begin{obs}{\bf Notation}\\
Let $\dbs{M}$ a bisimplicial set and let $\ubs{\mathcal{U}}$
be an open cover of $\dbs{M}$. In the settings of the proof
of lemma \ref{bscover}, given a pair of non negative integers $(m,n)$ we will denote the element $([m],[n])  \in \Delta^2$ by $(m,n)$. If $\lambda \in \Lambda_{m,n}$ then $\lambda$ satisfies tjhe following conditions:
\begin{enumerate}
\item $\lambda$ is a map from $\Pc_{m,n}$ to $\bigcup_{k,l} I_{k,l}\quad$,
\item for all pair $(k,l)$ of non negative integers, if $(k,l) \leq (m,n)$ then $\lambda(\Pc_{m,n}^{k,l}) \subseteq I_{k,l}\quad$.
\end{enumerate}
By definition of the category $\Delta'$, if $f \in \Pc_{m,n}^{k,l}$ then $f:= (f_1, f_2) $ with $f_1: [k] \to [m]$ and $f_2: [l] \to [n]$ are strictly  increasing. It follows that they are one to one and hence $f_{1}([k])$ and $f_{2}([l])$ are subsets of $[m]$ and $[n]$, respectively, of cardinality $k+1$ and $l+1$, respectively.

The above reasoning permit us to identify any map $f \in \Pc_{m,n}$ with a pair of non empty subsets of $\Pc([m])$ and $\Pc([n])$, and, therefore, any $\lambda \in \Lambda_{m,n}$ with a matrix array 
$[\lambda_{S,T}]$ of size $(2^{m+1}-1)(2^{n+1}-1)$ with $S$ and $T$ varying over $\Pc([m])-\{\emptyset\}$ and $\Pc([n])-\{\emptyset\}$, respectively, and $\lambda_{S,T}$ stands for $\lambda(S,T)$. Here we can order the pair of subsets $(S, T)$ by a lexicographic like order in the following way:
\begin{equation}\label{lexy_order}
(S_1, T_1) \preccurlyeq (S_2, T_2)\quad \text{if and only if}\quad
\begin{cases}
|S_1| \leq |S_2|, & \\
|S_1| = |S_2|\quad  \text{and} \quad S_1 \preccurlyeq_L S_2, & \\
|S_1| = |S_2|\quad  \text{and} \quad |T_1| < |T_2| &  \\
|S_1| = |S_2|,\quad |T_1|=|T_2| \quad \text{and} \quad T_1 \preccurlyeq_L T_2. &
\end{cases}
\end{equation}
where $\preccurlyeq_L$ stands for the lexicographic order when $S$ and $T$ are displayed in increasing order.
\end{obs}
\begin{exa}
If $\lambda \in \Lambda_{1,1}$ then $\lambda = [\lambda(S, T)]_{3,3}$
and can be displayed as the matrix array
$$
\left [
\begin{matrix}
\lambda_{0,0}&\lambda_{0,1}&\lambda_{0,01}\\
\lambda_{1,0}&\lambda_{1,1}&\lambda_{1, 01}\\
\lambda_{01,0}&\lambda_{01,1}&\lambda_{01,01}\\
\end{matrix}
\right ],
$$
where, for any pair of non empty subsets $i \subseteq [m]$ and $j \subseteq [n]$, we have wrote $\lambda_{i,j} = \lambda(i,j)$.
\end{exa}

The family of all covers of a bisimplicial space $\dbs{M}$
is endowed with a partial preorder. 

\begin{definition}\label{order_of_coverings}
Let $\dbs{M}$ be a bisimplicial topological space and suppose that $\dbs{\Uc}$ and $\dbs{\Vc}$ are open covers of $\dbs{M}$, with $\Uc_{(m,n)} = \{ U^{m,n}_i \}_{i \in I_{m,n}}$ and $\Vc_{(m,n)} = \{ V^{m,n}_j\}_{j \in J_{m,n}} $. We said that $\dbs{\Vc}$ is finer than $\dbs{\Uc}$ if there is a family of maps $\theta_{m,n} : J_{m,n} \to I_{m,n}$ such that $V^{m,n}_j \subseteq U^{m,n}_{\theta(j)}$,
for all $j \in J_{m,n}$. 

If the covers are bisimplicial we take the maps $\dbs{\theta}$ bisimplicial.
\end{definition}

\subsection{\v{C}ech cohomology of double groupoids}

Let $\dbs{M}$ be a bisimplicial space, $\dbs{\Uc}$ be a 
bisimplicial open cover of $\dbs{M}$ and $\ubs{\F}$ be
a bisimplicial abelian sheaf. Let us define 
\begin{equation}
C^{m,n}_{bs}(\dbs{\Uc}, \ubs{\F}) = \prod_{i \in I_{m,n}}{\F^{m,n}(U^{m,n}_i)},
\end{equation}
and  horizontal and vertical differentials given by
\begin{align}
d_h^{m,n}: C^{m,n}_{bs}(\dbs{\Uc}, \ubs{\F}) \to C^{m,n+1}_{bs}(\dbs{\Uc}, \ubs{\F})
\quad \text{with} \quad (d_h^{m,n}c)_i = 
\sum_{k = 0}^{n+1} (-1)^k \widetilde{\epsilon^{m,n+1}_{k,h}}^\ast c_{\widetilde{\epsilon^{m,n+1}_{k,h}}(i)},\\
d_v^{m,n}: C^{m,n}_{bs}(\dbs{\Uc}, \ubs{\F}) \to C^{m+1,n}_{bs}(\dbs{\Uc}, \ubs{\F})
\quad \text{with} \quad (d_v^{m,n}c)_i =
\sum_{k = 0}^{m+1} (-1)^k \widetilde{\epsilon^{m+1,n}_{k,v}}^\ast c_{\widetilde{\epsilon^{m+1,n}_{k,v}}(i)}.
\end{align}
The $\ast$ as a superscript is explained in definition \ref{bisim-sheaf} and
$\widetilde{\epsilon^{m,n+1}_{k,h}}^\ast c_{\widetilde{\epsilon^{m,n+1}_{k,h}}(i)}$ is the \textit{restriction}
of the section $c_{\widetilde{\epsilon^{m,n+1}_{k,h}}(i)} \in \F^{m,n}(U^{m,n}_{\widetilde{\epsilon^{m,n+1}_{k,h}}(i)})$ to a section in $\F^{m, n+1}(U^{m,n+1}_i)$, likewise 
$\widetilde{\epsilon^{m+1,n}_{k,v}}^\ast c_{\widetilde{\epsilon^{m+1,n}_{k,v}}(i)}$ is the restriction
of the section $c_{\widetilde{\epsilon^{m+1,n}_{k,v}}(i)} \in \F^{m,n}(U^{m,n}_{\widetilde{\epsilon^{m+1,n}_{k,v}}(i)})$ to a section in $\F^{m, n+1}(U^{m,n+1}_i)$.

It is clear that $(d_h)^2 = 0$ and, with the usual sign trick, the vertical differential 
satisfies the relations $ ((-1)^n d_v^{m,n})^2 = 0$ and $d_h^{m,n} + (-1)^{n+1}d_v^{m,n+1} = (-1)^n d_v^{m,n} + d_h^{m+1,n} = 0$.
From this remark the collection $\{C_{bs}^{m,n}\}_{m,n \in \mathbb{N}}$ is a well defined double complex and we may consider the associated total complex $\Tot (\dbs{\mathcal{U}}, \ubs{\mathcal{F}})$  that allow us to define the total cohomology groups 
$\Ho_{\Tot}^{m}(\dbs{\mathcal{U}}, \ubs{\mathcal{F}})$.

\begin{definition}\label{bisimplicial_cohomology_groups}
Let $\dbs{M}$ be a bisimplicial space, $\dbs{\Uc}$ be a 
bisimplicial open cover of $\dbs{M}$ and $\ubs{\F}$ be
a bisimplicial abelian sheaf. The $n$-th \emph{bisimplicial cohomology group}
of $\dbs{M}$ with coefficients in $\ubs{\F}$ is defined as the direct limit
\begin{equation}
\check{\Ho}^n(\dbs{M}; \ubs{\F}) : = \underset{\longrightarrow}{\lim} \Ho^n_{\Tot} (\dbs{\mathcal{U}}; \ubs{\F}),
\end{equation}
where $\dbs{\mathcal{U}}$ runs over all open covers of $\dbs{M}$ whose $N$-skeleton 
admits an $N$-truncated simplicial structure for some $N \geq n+1$.
\end{definition}

The most important bisimplicial sheaf for our purposes
is the one constructed from a double groupoid acting along a map
into the total base of the double groupoid.

\begin{definition}
Let $\T = (\B; \Vc, \Hc; \Pc )$ be a topological double groupoid.
Then a $\T$-module  is a bundle of topological abelian groups
$p: \Kc \to \Pc$  such that
\begin{enumerate}
\item $\Kc$ is endowed with a left $\T$-action (see \ref{double groupoid left action}),
\item $\Vc$ and $\Hc$ acts on $\Kc$ by group bundle automorphisms.
\end{enumerate}
\end{definition}

\begin{remark}\label{associated_simplicial_to_action}
Let $\T = (\F; \Vc, \Hc; \Pc )$ be a topological double groupoid,
and let $\gamma: \Kc \to \Pc$  be a $\T$-module. Let us denote
by $p_{m,n}: \F^{m,n} \to \Pc$ the map $(F_{ij}) \mapsto p_{m,n}(F_{ij}) = bl(F_{m,1})$
and let $\F_{m,n} = \Kc \pfib{\gamma}{p_{m,1}} \F^{m,n}$.
This construction generate a new bisimplicial set,
indeed, with any $f: ([m], [n]) \to ([k], [l])$ in $\Delta^2$ 
we associate a map $\tilde{f}: \F_{k,l} \to \F_{m,n}$ given by
\begin{equation*}
\tilde{f}(K,\langle F_{i,j}\rangle_{0 \leq i \leq k , \; 0 \leq j \leq l} ) = 
(K, \langle F_{f(i,j)} \rangle_{1 \leq i \leq m , \; 1 \leq j \leq n})
\end{equation*}
It is no difficult to show that with this map we have a contravariant functor 
from $\Delta^2$ to $\mathcal{S}ets$.

\end{remark}

\begin{obs}\label{associated_sheaf}
In the discrete case we can recover the double groupoid cohomology, defined in \ref{discrete_double_groupoid_cohomology}, from the bisimplicial cohomology  \ref{bisimplicial_cohomology_groups} just defined. In fact, let us consider a topological double groupoid $\T= (\B,;\Hc, \Vc; \Pc)$ and suppose that $\ubs{\F}$ is the bisimplicial sheaf associated to a $\T$-module. 
From \ref{associated_simplicial_to_action}, we can construct
a bisimplicial set $\{ \Kc \pfibrado{p}{\gamma} \F^{m,n} \}$ and then, the projection map
\begin{equation}
\Pi^{m,n}:\Kc \pfibrado{p}{\gamma} \F^{m,n} \to \F^{m,n},
\end{equation}
is a bisimplicial map. If we denote by $\A^{m,n}$ the sheaf of germs of local continuous sections of $\Pi^{m,n}$, then $\ubs{\A} = \{\A^{m,n}\}_{m,n}$ becomes a bisimplicial sheaf and any section of $\Pi^{m,n}$ defined over $U \subseteq \F^{m,n}$, an open set,  can be identified with a continuous map
$\varphi:U \to \Kc$ such that $p(\varphi(\langle F_{i,j} \rangle)) = \gamma(\langle F_{i,j} \rangle)$
for any $\langle F_{i,j} \rangle \in \F^{m,n}$.  Under this bisimplicial structure, the maps $\epsilon_h^{m,n+1}:\A^{m,n} \to \A^{m,n+1}$ and $\epsilon_v^{m+1,n}:\A^{m,n} \to \A^{m+1,n}$ can be described as follows.

Given a local section $\varphi:U \to \Kc$ of $\Pi^{m,n}$
we can write $\varphi$ as
\begin{equation}
\varphi([F_{i,j}]) =l(F_{m1})^{-1} \cdots l(F_{21})^{-1} l(F_{11})^{-1} \varphi_v([F_{i,j}])
\end{equation}
or as 
\begin{equation}
\varphi([F_{i,j}]) =b(F_{m1}) b(F_{m2})  \cdots b(F_{mn}) \varphi_h([F_{i,j}])
\end{equation}
for unique $\varphi_v([F_{ij}]) \in \Kc_{tl(F_{m1})}$ and  $\varphi_h([F_{ij}]) \in \Kc_{rb(F_{mn})}$. Then $\epsilon_v^{m+1,n}\varphi$ is a germ of a local section of $\Pi ^{m+1,n}$ such that, for any $[F_{ij}] \in \F^{m+1,n}$ we have
\begin{equation}
(\epsilon_{0,v}^{m,n}\varphi)[F_{i,j}]=
l(F_{m+1,1})^{-1}l(F_{m1})^{-1} \cdots l(F_{31})^{-1} l(F_{21})^{-1} \varphi_v\begin{pmatrix}
F_{21}  & \dots & F_{2n} \\
\dots  & \dots & \dots \\
F_{m1}  & \dots & F_{mn} \\
F_{m+1,1}  & \dots & F_{m+1,n}
\end{pmatrix},
\end{equation}
if $0 < k < m+1$ then
\begin{equation}
(\epsilon_{k,v}^{m,n}\varphi)[F_{i,j}]=
l(F_{m+1,1})^{-1}l(F_{m1})^{-1} \cdots l(F_{21})^{-1} l(F_{11})^{-1} \varphi_v \begin{pmatrix}
F_{11}  & \dots & F_{1n} \\
\dots  & \dots & \dots \\
\left\{\begin{matrix}F_{k1} \\F_{k+1,1}\end{matrix}\right\}
& \dots & \left\{\begin{matrix}F_{kn} \\F_{k+1,n}\end{matrix}\right\} \\
\dots  & \dots & \dots \\
F_{m+1,1}  & \dots & F_{m+1,n}
\end{pmatrix}
\end{equation}
and 
\begin{equation}
(\epsilon_{m+1,v}^{m,n}\varphi)[F_{i,j}]=
l(F_{m+1,1})^{-1}l(F_{m1})^{-1} \cdots l(F_{21})^{-1} l(F_{11})^{-1} \varphi_v\begin{pmatrix}
F_{11}  & \dots & F_{1n} \\
\dots  & \dots & \dots \\
F_{m1}  & \dots & F_{mn} \end{pmatrix};
\end{equation}
in the same way, for any $[F_{ij}] \in \F^{m,n+1}$, the sections corresponding to the horizontal maps 
$\epsilon_{k,h}^{m,n}$ are given by
\begin{equation}
(\epsilon_{0,h}^{m,n}\varphi)[F_{i,j}]=
b(F_{m1})b(F_{m2}) \cdots b(F_{mn})b(F_{m,n+1}) \varphi_h\begin{pmatrix}
F_{12}  & \dots & F_{1,s+1} \\
\dots  & \dots & \dots \\
F_{m2}  & \dots & F_{m,n+1}
\end{pmatrix},
\end{equation}
if $0 < k < n+1$ then
\begin{equation}
(\epsilon_{k,h}^{m,n}\varphi)[F_{i,j}]=
b(F_{m1})b(F_{m2}) \cdots b(F_{mn})b(F_{m,n+1}) \varphi_h \begin{pmatrix}
F_{11}  & \dots & \left\{F_{1k} F_{1,k+1} \right\}& \dots & F_{1,n+1} \\
\dots  & \dots& \dots& \dots & \dots \\
F_{m,1}  & \dots& \left\{F_{mk}F_{m,k+1} \right\}& \dots & F_{m,n+1}
\end{pmatrix},
\end{equation}
and 
\begin{equation}
(\epsilon_{m+1,h}^{m,n}\varphi)[F_{i,j}]=
b(F_{m1})b(F_{m2}) \cdots b(F_{mn}) \varphi_h\begin{pmatrix}
F_{11}  & \dots & F_{1n} \\
\dots  & \dots & \dots \\
F_{m1}  & \dots & F_{mn} \end{pmatrix}.
\end{equation}
If we use the above expressions, together with the formulas for the horizontal and vertical co\-boun\-da\-ry map, and if we consider the largest open bisimplicial covering 
$\{ \mathcal{U}^{m,n}_{F}\}_{F \in \F^{m,n}}$ of $ \F^{m,n}$, where $ \mathcal{U}^{m,n}_{F}=\{F\}$ for every $F \in \F^{m,n}$, then we can easily see that the \v{C}ech cohomology groups coincides with the cohomology groups of the double groupoid cohomology introduced in \ref{discrete_double_groupoid_cohomology}.
\end{obs}

\subsection{Low dimensional cohomology}

\begin{definition}\label{extensions}
Let $(\B; \Vc, \Hc; \Pc)$ be a topological double groupoid and $\Kc \to \Pc$ an abelian group bundle. We define $Ext(\B, A)$ to be the set 
\begin{equation}\label{extension_form}
Ext(\B, \Kc) = 	\underset{\mathcal{U}}  {\underset{\longrightarrow} \lim } \; \Opext(\B[\mathcal{U}], \Kc[U])
\end{equation}
Where  $\mathcal{U} = \{ \mathcal{U}_i \}_{i \in I}$ runs over open
covers of $\Pc$ and $(\mathcal{B}[\mathcal{U}]); \Vc[\mathcal{U}], \Hc[\mathcal{U}]; \Pc[\mathcal{U}])$
is the \textit{\v{C}ech double groupoid associated to $(\B, \mathcal{U})$}.

\end{definition}

\begin{theorem}\label{class_ext_simplicial_version}
Let $(\F; \Vc, \Hc; \Pc)$ be a topological double groupoid and $p: \Kc \to \Pc$ 
be an abelian group bundle, and $\ubs{\mathcal{A}}$ the bisimplicial sheaf associate to the action of $\F$ over $\Kc$.
\begin{itemize}
\item For each open cover $\dbs{\mathcal{U}}$ of $\ubs{\F}$, there is a canonical isomorphism
\begin{equation}\label{extCech}
\Opext_{\Uc} (\F[\Uc_{00}], \Kc[\Uc_{00}]) \cong \Ho^{1}_{\Tot}(\dbs{\Uc};\ubs{\A}),
\end{equation}
where $\Opext_{\Uc} (\F[\Uc_{00}], \Kc[\Uc_{00}])$ denotes the subgroup of elements of $\Opext (\F[\Uc_{00}], \Kc[\Uc_{00}])$
consisting of extensions $1 \rightarrow \Kc[\Uc_{00}] \overset{\iota}\hookrightarrow \B \overset{\Pi}\twoheadrightarrow \F[\Uc_{00}] \to 1$
such that $\Pi$ admits a continuous lifting over each open set $U_i^{11} \subseteq \F$ ( $i \in I_{11}$).
\item The isomorphisms \ref{extCech} induces another one 
\begin{equation}
\Ext(\F, \Kc) \cong \check{\Ho}^{1}_{\Tot}(\dbs{\Uc}; \ubs{\A})
\end{equation}
\end{itemize}
\end{theorem}
\begin{proof}
Since we need a detailed study of simplicial covers and to give an explicit description of the two cocycles, the proof is divided in several stages.

{\bf Step 1.} \emph{Description of total one cocycles.}\\
If $\dbs{\mathcal{U}}$ is an open cover of $\dbs{\F}$, we know by definition that the double complex which gives rise to the C\v{e}ch cohomology of the double groupoid is
\begin{equation}\label{complex_simplicial_double_cohomology}
\xymatrix{
&\vdots&\\
&\C^{3,1}(\dbs{bs(\Uc)},\ubs{\A}) \ar[u]^{d_v^{3,1}} \ar[r]_{d_h^{3,1}}& \C^{3,2}(\dbs{bs(\Uc)},\ubs{\A}) \\
 &\C^{2,1}(\dbs{bs(\Uc)},\ubs{\A}) \ar[u]^{d_v^{2,1}} \ar[r]_{d_h^{2,1}} & \C^{2,2}(\dbs{bs(\Uc)},\ubs{\A}) \ar[r]_{d_h^{2,2}} \ar[u]^{d_v^{2,2}}& \C^{2,3}(\dbs{bs(\Uc)},\ubs{\A})\\
&\C^{1,1}(\dbs{bs(\Uc)},\ubs{\A}) \ar[u]^{d_v^{1,1}}\ar[r]_{d_h^{1,1}} & \C^{1,2}(\dbs{bs(\Uc)},\ubs{\A}) \ar[u]^{d_v^{1,2}}
\ar[r]_{d_h^{1,2}} & \C^{1,3}(\dbs{bs(\Uc)},\ubs{\A}) \ar[r]_{\quad \quad\quad d_h^{1,3}}\ar[u]^{d_v^{1,3}}& \cdots
}
\end{equation}
Hence, since the open cover $bs(\dbs{\Uc})$ is the bisimplicial refinement of $\dbs{\Uc}$ defined in \ref{bscover},the first terms of the total chain complex are
\begin{align}
0 \to \Tot^1(bs(\dbs{\Uc}), \ubs{\A}) \to \Tot^2(bs(\dbs{\Uc}), \ubs{\A})\to \Tot^3(bs(\dbs{\Uc}), \ubs{\A}) \to \ldots
\end{align}
where
\begin{align}
&\Tot^1(bs(\dbs{\Uc}), \ubs{\A})=\prod\limits_{\lambda \in \Lambda_{1,1}} \A^{1,1}(U^{1,1}_{\lambda}),\\
&\Tot^2(bs(\dbs{\Uc}), \ubs{\A})= \prod\limits_{\lambda \in \Lambda_{2,1}} \A^{2,1}(U^{2,1}_{\lambda}) \oplus \prod\limits_{\lambda \in \Lambda_{1,2}} \A^{1,2}(U^{1,2}_{\lambda}),\\
&\Tot^3(bs(\dbs{\Uc}), \ubs{\A})= \prod\limits_{\lambda \in \Lambda_{3,1}} \A^{3,1}(U^{3,1}_{\lambda}) \oplus \prod\limits_{\lambda \in \Lambda_{2,2}} \A^{2,2}(U^{2,2}_{\lambda}) \oplus \prod\limits_{\lambda \in \Lambda_{1,3}} \A^{1,3}(U^{1,3}_{\lambda}).
\end{align}

A two cocycle in $Z^2(\dbs{bs(\Uc)},\ubs{\A})$ is a pair 
$(\sigma, \tau)$ where $\sigma$ and $\tau$ are families of the form $\sigma=(\sigma_\lambda)_{\lambda \in \Lambda_{2,1}}$ and $\tau = (\tau_{\lambda})_{\lambda \in \Lambda_{1,2}}$ such that $d_{\Tot}^2(\sigma, \tau) = 0$.  This equation is equivalent to
\begin{align*}
d_v^{2,1}(\sigma) &= 0,\\
d_h^{2,1}(\sigma) + d_v^{1,2}(\tau) &= 0,\\
d_h^{1,2}(\tau) &= 0;
\end{align*}
which amounts to the following three ones
\begin{equation}\label{two_cocycle_1}
\sum_{k=0}^3 (-1)^k \widetilde{\epsilon_{k,v}^{3,1}}^{*}(\sigma_{\widetilde{\epsilon_{k,v}^{3,1}}(\lambda)}) = 0, \quad\text{for any} \quad \lambda \in \Lambda_{3,1};
\end{equation}
\begin{equation}\label{two_cocycle_2}
\sum_{k=0}^{3}(-1)^k\widetilde{\epsilon_{k,h}^{2,2}}^{*} (\sigma_{\widetilde{\epsilon^{2,2}_{k,h}}}) + \sum_{k=0}^{3}(-1)^k\widetilde{\epsilon_{k,v}^{2,2}}^{*}(\tau_{\widetilde{\epsilon_{k,v}^{2,2}}(\lambda)}) =0, \quad  \text{for any} \quad \lambda \in \Lambda_{2,2}, \quad \text{and}
\end{equation}
\begin{equation}\label{two_cocycle_3}
\sum_{k=0}^{3}(-1)^{k}\widetilde{\epsilon_{k,h}^{1,3}}^{*}(\tau_{\widetilde{\epsilon_{k,h}^{1,3}}(\lambda)}) = 0, \quad \text{for any} \quad \lambda \in \Lambda_{1,3}.
\end{equation}
In order to obtain more concrete information about the two cocycles, we will now analyze each of those equations .

The equation \ref{two_cocycle_1} is valid in the open set
$$U_{\lambda}^{3,1}=\bigcap_{(l,k)\leq (3,1)} \bigcap_{f \in \Pc_{3,1}^{l,k}} f^{-1}(U_{f(\lambda)}^{3,1}),$$
of $\F^{(3,1)}$ and if we apply $\widetilde{\eta^{3,1}_{0,v}}^{\ast}$ to both sides of it, we obtain 
\begin{equation}
\sigma_{\lambda_{[3]\setminus 0,[1]}}= \widetilde{\eta^{3,1}_{0,v}\epsilon^{3,1}_{1,v}}^{*}(\sigma_{\lambda_{[3]\setminus 1,[1]}}) - 
\widetilde{\eta^{3,1}_{0,v}\epsilon^{3,1}_{2,v}}^{*}(\sigma_{\lambda_{[3]\setminus 2,[1]}}) + 
\widetilde{\eta^{3,1}_{0,v}\epsilon^{3,1}_{3,v}}^{*}(\sigma_{\lambda_{[3]\setminus 3,[1]}}).
\end{equation}
If we write down $\lambda_{[3]\setminus 0,[1]}, \lambda_{[3]\setminus 1,[1]}, \lambda_{[3]\setminus 2,[1]}$ and $\lambda_{[3]\setminus 3,[1]}$ explicitly, it follows from the above equation that the section $\sigma$ is independent of the last row of the index $\lambda$. Then, since $\ubs{\A}$ is a family of sheaves then there is a section 
$$
\sigma_{\lambda_{S, T}} \in 
\A^{2,1}\left(\bigcap_{(k,l)\leq (2,1)} \bigcap_{f \in \Pc_{3,1}^{k,l}} f^{-1}(U_{f(\lambda)}^{3,1})\right),
$$
where $S \subseteq [2], T \subseteq [1]$ with $|S| \leq 2$, and such that $\sigma_{\lambda}$ is the restriction to $U^{3,1}_\lambda$ of 
$\sigma_{\lambda_{S, T}} \in \Lambda_{2,1}$.

Now, given a tuple  $\left( \begin{matrix} A \\ B \\ C \end{matrix} \right) \in \bigcap_{(l,k)\leq (2,1)} \bigcap_{f \in \Pc_{3,1}^{l,k}} f^{-1}(U_{f(\lambda)}^{3,1})$, and according to \ref{associated_sheaf}, we have
\begin{equation*}
\widetilde{\epsilon^{3,1}_{0,v}}^{*}(\sigma_{\widetilde{\epsilon^{3,1}_{0,v}}(\lambda)})\left( \begin{matrix} A \\ B \\ C \end{matrix} \right) =\sigma_{\lambda_{[3]\setminus 0,[1]}}\left( \begin{matrix} A \\ B \end{matrix} \right),  \quad
\widetilde{\epsilon^{3,1}_{1,v}}^{*}(\sigma_{\widetilde{\epsilon^{3,1}_{1,v}}(\lambda)})\left( \begin{matrix} A \\ B \\ C \end{matrix} \right) = \sigma_{\lambda_{[3]\setminus 1,[1]}}\left( \begin{matrix} \left\{ \begin{matrix} A \\ B \end{matrix} \right\} \\ C \end{matrix} \right),
\end{equation*}
\begin{equation*}
\widetilde{\epsilon^{3,1}_{2,v}}^{*}(\sigma_{\widetilde{\epsilon^{3,1}_{2,v}}(\lambda)})\left( \begin{matrix} A \\ B \\ C \end{matrix} \right) = \sigma_{\lambda_{[3]\setminus 2,[1]}}\left( \begin{matrix} A \\ \left\{ \begin{matrix} B \\ C \end{matrix} \right\} \end{matrix} \right)\quad \text{and} \quad
\widetilde{\epsilon^{3,1}_{3,v}}^{*}(\sigma_{\widetilde{\epsilon^{3,1}_{3,v}}(\lambda)})\left( \begin{matrix} A \\ B \\ C \end{matrix} \right) =  l(C)^{-1} \cdot \sigma_{\lambda_{[3]\setminus 3,[1]}}\left( \begin{matrix} B \\ C \end{matrix} \right),
\end{equation*}
and the equation \ref{two_cocycle_1} can be rewritten as
\begin{equation}\label{re_two_cocycle_1}
\sigma_{\lambda_{[3]\setminus 0,[1]}}\left( \begin{matrix} A \\ B \end{matrix} \right) - \sigma_{\lambda_{[3]\setminus 1,[1]}}\left( \begin{matrix} \left\{ \begin{matrix} A \\ B \end{matrix} \right\} \\ C \end{matrix} \right) + \sigma_{\lambda_{[3]\setminus 2,[1]}}\left( \begin{matrix} A \\ \left\{ \begin{matrix} B \\ C \end{matrix} \right\} \end{matrix} \right)-  l(C)^{-1}  \cdot \sigma_{\lambda_{[3]\setminus 3,[1]}}\left( \begin{matrix} B \\ C \end{matrix} \right) =0.
\end{equation}
In the same way we can deduce similar expressions for equations \eqref{two_cocycle_2} and \eqref{two_cocycle_3}.
 
\vspace{0.5 cm} 
{\bf Step 2.} \textit{Passing to a coarser covering of the double groupoid.}\\
Let us consider an open cover $\dbs{\Wc}$ of $\ubs{\F[\Uc]}$ defined in the following way:
\begin{itemize}
\item The indexing family $J_{m,n}$ is defined by 
$J_{0,0} = \{ \ast \}$\;,\; 
$J_{1,0} = I_{0,0}^2 \times I_{1,0}$\;,\;
$J_{0,1}=I_{0,0}^2 \times I_{0,1}$\;,\;
$J_{1,1} = I_{0,0}^2 \times I_{0,0}^2 \times I_{1,1}$\;, \; $J_{20} = I_{00}^3 \times I_{20}$\;, \;$J_{02}=I_{00}^3 \times I_{20}$\;,\;  and $J_{m,n}$ any indexing set in other cases.
\item The collection $\Wc_{m,n}=\{ W^{m,n}_j \}_{j \in J_{m,n}}$ is defined by
\begin{align*}
&W^{0,0}\quad \text{consist of only one open set} \quad \coprod_{i \in I_{0,0}} U^{0,0}_i \quad \text{(disjoint union);}\\
&W^{1,0}_{ijk} = 
\left\{  
\left(
\begin{matrix}
i\\
g\\
j
\end{matrix}
\right) 
\mid
g \in U^{10}_{k},\; t(g) \in U^{00}_{i} \;,\; b(g) \in U^{10}_{j} 
\right\} \; \text{for all $i,j \in I_{00}$ and $k \in I_{10}$;}\\
&W^{01}_{ijk} = 
\left\{  
\left(
\begin{matrix}
i & x & j
\end{matrix}
\right) 
\mid
x \in U^{01}_{k},\; l(x) \in U^{00}_{i} \;,\; r(x) \in U^{10}_{j} 
\right\} \; \text{for all $i,j \in I_{00}$ and $k \in I_{01}$;}\\
&W_{i_{11}i_{12}i_{21}i_{22}j}^{1,1} = 
\left\{
\left(
\begin{matrix}
i_{11} &   & i_{12} \\
	   & F & 		\\
i_{21} &   & i_{22}
\end{matrix}
\right)
\mid
B \in U^{1,1}_{j},\; tl(B) \in U^{0,0}_{i_{11}},\; tr(B) \in U^{0,0}_{i_{12}},\; bl(B) \in U^{0,0}_{i_{21}},\; br(B) \in U^{0,0}_{i_{22}}
\right\}
\end{align*}
for all $i_{00},i_{01},i_{10},i_{11} \in I_{00}$ and $j \in I_{11}$; and let $V^{m,n}$ be any open covering indexed by $J_{m,n}$ , for any other pair $(m,n) \in \mathbb{N}^2$.
\end{itemize}
We will to calculate the first total cohomology group of the pullback sheaf $\ubs{\mathcal{S}}$ of $\ubs{\A}$ 
along the natural  projection map $\ubs{P}: \ubs{\F[\Uc]} \to \ubs{\F}$. 
More exactly, we are going to show that $\Hom_{\Tot}^1(\dbs{\Vc},\ubs{\mathcal{S}})$ is isomorphic 
to $\Hom_{\Tot}^1(\dbs{\Uc}, \ubs{\A})$.

Let us to denote by $\Gamma=\{\Gamma_{m,n}\}$ the indexing family obtained by the process of the proof of lemma \ref{bscover}, applied to the cover $\Vc$ indexed by the family $J$. Here, every $\gamma \in \Gamma_{2,1}$ is a map $\gamma: \Pc_{2,1} \to \bigcup_{k,l} J_{k,l}$
such that $\gamma(\Pc_{2,1}^{k,l}) \subseteq J_{k,l}$ for every $(k,l) \leq (2,1)$,
and it can be represented as a matrix array of size $7 \times 3$, where the rows are indexed by non empty subsets of $[2]$ and the columns are indexed by non empty subsets of $[1]$. Moreover, by definition of $J$, the value $\gamma(S,T) = \ast$ if $S = 0,1, 2$ and $T = 0, 1$.
That is, the block
\begin{equation*}
\left[
\begin{matrix}
 \gamma_{0,0} &  \gamma_{0,1} \\
 \gamma_{1,0} &  \gamma_{1,1} \\
 \gamma_{2,0} &  \gamma_{2,1} \\
\end{matrix} 
\right]
=
\left[
\begin{matrix}
\ast & \ast \\
\ast & \ast \\
\ast & \ast  
\end{matrix}
\right].
\end{equation*}

In the same way, every $\gamma \in \Gamma_{1,2}$ can be represented by a matrix array $\gamma=[\gamma(S,T)]$, of size $3 \times 7$,  where the rows are indexed by non empty subsets $S \subseteq [1]$ and the columns are indexed by non empty subsets $T \subseteq [2]$. As for $\Gamma_{2,1}$, we have that
\begin{equation*}
\left[
\begin{matrix}
 \gamma_{0,0} &  \gamma_{0,1} &  \gamma_{0,2} \\
 \gamma_{1,0} &  \gamma_{1,1} &  \gamma_{1,2} \\
\end{matrix} 
\right]
=
\left[
\begin{matrix}
\ast & \ast & \ast \\
\ast & \ast & \ast \\
\end{matrix}
\right].
\end{equation*}

Let $\gamma:\Lambda_{21} \to \Gamma_{21} $ be the map defined as follows. Given $\lambda \in \Lambda_{21}$ we define $\gamma(\lambda) \in \Gamma_{21}$  by the rules
\begin{itemize}
\item $\gamma(\lambda)_{ST} = \ast$ if $S = 0,1$ or $2$, and $T = 0$ or $1$;
\item $\gamma(\lambda)_{ST} = (\lambda_{ik},\lambda_{jk},\lambda_{ST})$ if $S=\{ i, j\}$ and $T= k$;
\item $\gamma(\lambda)_{S,01} = (\lambda_{S0}, \lambda_{S1}, \lambda_{S,01})$ for $S=0,1$ or $2$;
\item $\gamma(\lambda)_{S,01}=(\lambda_{i0},\lambda_{i1},\lambda_{j0},\lambda_{j1},\lambda_{S,01})$ if $S=\{i,j\}$;
\item $\gamma(\lambda)_{012,T}=(\lambda_{0T},\lambda_{1T},\lambda_{2T},\lambda_{012,T})$ if $T=0$ or $1$;
\item $\gamma(\lambda)_{012,T}=(\lambda_{00},\lambda_{01},\lambda_{10},\lambda_{11},\lambda_{20},\lambda_{21},\lambda_{012,01})$.
\end{itemize}
From the definition of the indexing family $J$ it is clear that $\gamma$ is a bijective map with inverse denoted by $\lambda$. In the same way we can define a bijective map $\Lambda_{12} \to \Gamma_{12}$ which we also denote by $\gamma$ (and inverse $\lambda$) and from the context it will be clear which of them we are using.

Let $\Xi: Z^1(\dbs{\Vc},\ubs{\mathcal{S}}) \to Z^1(\dbs{\Uc},\ubs{\A})$ be the map defined by 
$$(\varphi,\psi)=(\{\varphi_\gamma\}_{\gamma \in \Gamma_{21}},\{\psi_{\gamma}\}_{\gamma \in \Gamma_{12}}) \mapsto (\Xi_1(\varphi),\Xi_2(\psi)):= (\sigma, \tau),
$$
with $\sigma=\{\sigma_{\lambda}\}_{\lambda \in \Lambda_{21}}$ and $\tau= \{\tau_{\lambda}\}_{\lambda \in \Lambda_{12}}),$ where $\sigma_{\lambda}:= \varphi_{\gamma(\lambda)}$ and $\tau_{\lambda} := \psi_{\gamma(\lambda)} $ for all $\lambda$ in $\Lambda_{21}$ or in $\Lambda_{12}$, respectively.

Given a pair $(\varphi, \psi) \in Z^1(\dbs{\Vc}, \ubs{\mathcal{S}})$, then
$\varphi$ is a family $(\varphi_\gamma)_{\gamma \in \Gamma_{2,1}}$ such that for any $\gamma \in \Gamma_{2,1}$
the following equation holds
\begin{equation}\label{coarser_two_cocycle_1}
l(C)^{-1} \varphi_{\omega_{[3]\setminus 0,[1]}}\left( \begin{matrix} A \\ B \end{matrix} \right) 
- \varphi_{\omega_{[3]\setminus 1,[1]}}\left( \begin{matrix} \left\{ \begin{matrix} A \\ B \end{matrix} \right\} \\ C \end{matrix} \right)
+ \varphi_{\omega_{[3]\setminus 2,[1]}}\left( \begin{matrix} A \\ \left\{ \begin{matrix} B \\ C \end{matrix} \right\} \end{matrix} \right)
- \varphi_{\omega_{[3]\setminus 3,[1]}}\left( \begin{matrix} B \\ C \end{matrix} \right) =0.
000\end{equation}

Since $(\varphi, \psi)$ satisfies the cocycle conditions, it is clear that the map $\Xi$ is well defined, moreover it is an isomorphism of abelian groups. The relations
\begin{equation}
(\Xi_1 d^1 \varphi)_\lambda = (d^1 \varphi)_{\gamma(\lambda)} \quad \text{and} \quad 
(\Xi_2 d^1 \psi)_\lambda = (d^1 \psi)_{\gamma(\lambda)},
\end{equation}
satisfied by the map $\Xi$, allow us to induce an isomorphism between the total cohomology groups
$\Hom_{\Tot}^1(\dbs{\Vc},\ubs{\mathcal{S}})$ and $\Hom_{\Tot}^1(\dbs{\Uc}, \ubs{\A})$.

{\bf Step 3.} \textit{From double groupoid extensions to cohomology.}\\
The above result allow us to consider $\dbs{\Uc}$ as an open covering 
with $\Uc_{00} = \Pc$.

Let us consider an extension 
$$1 \to \Kc \overset{\iota}{\hookrightarrow} \B \overset{\Pi}{\twoheadrightarrow} \F \to 1$$
in $\Opext_{\Uc} (\F[\Uc_{00}], \Kc[\Uc_{00}])$. For any $\lambda \in \Lambda_{21}$, if 
$$
\left(\begin{matrix}
A \\ B
\end{matrix}
\right) \in U^{21}_{\lambda} = \bigcap_{(k,l) \leq (2,1)} \bigcap_{f \in \Pc^{k,l}_{21}} \tilde{f}^{-1}(U^{kl}_{\lambda(f)}),
$$
then, by considering the cases $f = \epsilon^{21}_{2,v}$,  $\epsilon_{0,v}^{21}$ and $\epsilon_{1,v}^{21}$ we can assert that $A \in U^{21}_{\lambda_{01}}$, $B \in U_{\lambda_{12}}^{21}$   and 
$
\left\{
\begin{matrix} A \\ B 
\end{matrix}
\right\} \in U^{21}_{\lambda_{02}}$. 

Since there are local sections $\mu_{\lambda_{01,01}}: U^{21}_{\lambda_{01,01}} \to \Kc$, $\mu_{\lambda_{02,01}}: U^{21}_{\lambda_{02,01}} \to \Kc $ and $\mu_{\lambda_{12,01}}: U^{21}_{\lambda_{12,01}} \to \Kc$ of $\Pi$, then we can define $\sigma_{\lambda}: U^{21}_{\lambda} \to \Kc$ by the equation
\begin{equation}\label{action_abelian_grup_bundle_1}
\left\{
\begin{matrix}
\mu_{\lambda_{01,01}}(A) \\
\mu_{\lambda_{12,01}}(B)
\end{matrix}
\right\}
= \sigma_{\lambda}
\left(
\begin{matrix}
A \\
B
\end{matrix}
\right) 
\rightharpoonup \mu_{\lambda_{02,01}}
\left\{
\begin{matrix}
A \\B
\end{matrix}
\right\}
\end{equation}

To show that the above equation defines a simplicial 2-cocycle for the cohomology of the
double groupoid, we come back to the associativity of horizontal and vertical composition, and to the exchange law between them. In fact, if $\lambda \in \Lambda_{31}$ and
$$
\left(
\begin{matrix}
A \\ B \\ C
\end{matrix}
\right)
\in U^{31}_{\lambda} = \bigcap_{(k,l) \leq (3,1)} \bigcap_{f \in \Pc^{(k,l)}_{(3,1)}} \tilde{f}^{-1}(U^{kl}_{\lambda(f)}),
$$
with $f = \epsilon_{i,v}^{31}$  for $i=0,1,2$ and $3$, we have
\begin{equation*}
\left(
\begin{matrix}
B \\ C
\end{matrix}
\right)
\in U^{21}_{\lambda_{[3] \setminus 0}}\;,\;
\left(
\begin{matrix}
\left\{
\begin{matrix}
A \\
B
\end{matrix}
\right\} \\
C
\end{matrix}
\right)
\in U^{21}_{\lambda_{[3] \setminus 1}}\;,\;
\left(
\begin{matrix}
A \\
\left\{
\begin{matrix}
B \\ C
\end{matrix}
\right\}
\end{matrix}
\right)
\in 
U^{21}_{\lambda_{[3] \setminus 2}}\;,\quad \text{and} \quad
\left(
\begin{matrix}
A \\ B
\end{matrix}
\right) \in U^{21}_{\lambda_{[3] \setminus 3}}.
\end{equation*}
Since the vertical composition law is associative, by modifying in each case equation \eqref{action_abelian_grup_bundle_1}, we  obtain that $\sigma = \{\sigma_{\lambda}\}_{\lambda \in \Lambda_{21}}$ satisfies the cocycle equation \eqref{re_two_cocycle_1}. In a similar way, by using horizontal composition law, instead of the vertical one, we can define $\{ \tau_{\lambda}\}_{\lambda \in \Lambda_{12}} \in \prod_{\lambda \in \Lambda_{12}} \mathcal{A}^{12}(U^{12}_{\lambda})$, and in the same way we can show that the other equation that defines a total $2$-cocycle are satisfied.
To see that the cohomology class $[(\sigma, \tau)]$ is independent of the local sections initially taken, we consider $\lambda \in \Lambda_{11}$ and another local sections $\mu_{\lambda_{01,01}}': U^{21}_{\lambda_{01,01}} \to \Kc$ of $\Pi$. Since $\mu_{\lambda_{01,01}}$ and $\mu_{\lambda_{01,01}}'$ have the same sides, there is a continuous map $\alpha_{\lambda}: U^{11}_{\lambda} \to \Kc$, such that for all $A \in U^{21}_{\lambda}$  the equation 
\begin{equation}\label{cohomologous_cocycles_1}
{\mu_{01,01}}(A) = \alpha_{\lambda}(A) \rightharpoonup \mu_{\lambda_{01,01}}'(A)
\end{equation}
holds. Similar equations hold for the other two local sections. Replacing \eqref{cohomologous_cocycles_1} in \eqref{action_abelian_grup_bundle_1} and comparing with the respective equation for $\mu_{\lambda_{01,01}}'$ we find 
\begin{equation}
(\sigma -\sigma')_{\lambda}
\left(
\begin{matrix}
A\\B
\end{matrix}
\right)
=
\alpha_{12,01}(B) - \alpha_{02,01}
\left\{ 
\begin{matrix}
A \\ B
\end{matrix}
\right\}
+ l(B)^{-1} \cdot \alpha_{01,01}(A) = (d_{v}^{11} \alpha)_{\lambda}
\left(
\begin{matrix}
A \\ B
\end{matrix}
\right).
\end{equation}
In the same way we show that $\tau$ satisfies a similar equation, but with the horizontal composition instead of the vertical one, and we can conclude that $(\sigma, \tau)$  and $(\sigma', \tau')$ are cohomologous.

{\bf Step 4.} \textit{From cohomology to double groupoid extensions.}\\
Let $(\sigma, \tau) \in \prod\limits_{\lambda \in \Lambda_{2,1}} \A^{2,1}(U^{2,1}_{\lambda}) \oplus \prod\limits_{\lambda \in \Lambda_{1,2}} \A^{1,2}(U^{1,2}_{\lambda})$ be a 1-cocycle of the double groupoid cohomology

We need to construct a new double double groupoid $\B$ from $(\sigma, \tau)$ that 
fit in an extension 
$$1 \to \Kc \overset{\iota}{\hookrightarrow} \B \overset{\Pi}{\twoheadrightarrow} \F \to 1$$
of $\F$ by $\Kc$ and such that $\Pi$ has sections over each open set $U^{11}_i$ in the cover $\Uc_{11}$.
Thus define 
$$\B' = \coprod_{i \in I_{11}} \left\{ (K, F, i) \mid K \in \Kc,\; F \in U^{11}_i \quad
\text{and} \quad p(K) = lb(F) \right\},$$ 
and let $\sim$ be the relation on $\B'$ defined by 
\begin{equation}\label{eq_rel_1}
(K, F, k) \sim (-\tau_{\lambda(i)}(\id_{l(F)},\id_{l(F)}) + K + \tau_{\lambda(ijk)}(\id_{l(F)}, F),F , j),
\end{equation}
where $\lambda(ijk)$ stands for an element $\lambda \in \Pc_{(1,2)}^{(1,1)}$ with $\lambda_{01,01}= i, \lambda_{01,02}=j$ 
and $\lambda_{01,12}=k$, and $(\id_l(F), \id_l(F)) \in U^{12}_{\lambda(i)}$ 
along with $(\id_{l(F)}, F) \in U^{12}_{\lambda(ijk)}$;   and
\begin{equation}\label{eq_rel_2}
(K, F, k) \sim \left(-\sigma_{\lambda(i)}\left(\begin{matrix}\id_{t(F)} \\ \id_{t(F)}\end{matrix}\right) + K 
+ \sigma_{\lambda(ijk)}\left(\begin{matrix} \id_{t(F)} \\ F \end{matrix} \right), F, j \right),
\end{equation}
where $\lambda(ijk)$ stands for an element $\lambda \in \Pc_{(2,1)}^{(1,1)}$ with $\lambda_{01,01}= i, \lambda_{02,01}=j$ 
and $\lambda_{12,01}=k$, and $\left( \begin{matrix} \id_{l(F)} \\ \id_{l(F)}\end{matrix}\right) \in U^{21}_{\lambda(i)}$ 
along with  $\left( \begin{matrix}\id_{l(F)} \\ F \end{matrix} \right) \in U^{12}_{\lambda(ijk)}$.

To show that $\sim$ defines an equivalence relation on $\B'$, let us define
\begin{align}
\psi_{ikj}^h(F) = -\tau_{\lambda(i)}(\id_{l(F)},\id_{l(F)}) + \tau_{\lambda(ijk)}(\id_{l(F)}, F), \label{eq_rel_3}\\
\psi_{ikj}^v(F) = -\sigma_{\lambda(i)}\left(\begin{matrix}\id_{t(F)} \\ \id_{t(F)}\end{matrix}\right) \label{eq_rel_4}
+ \sigma_{\lambda(ijk)}\left(\begin{matrix} \id_{t(F)} \\ F \end{matrix} \right).
\end{align}
By using cocycle conditions it is no difficult to show that $\psi_{ikj^v}$ and $\psi_{ikj}^h$
are independent of the value of $i$ and that
\begin{center}
\begin{tabular}{lll}
$\psi_{jj}^v = 0$, & & $\psi_{jj}^h = 0$; \\
$\psi_{kj}^v = -\psi_{jk}^v$, & & $\psi_{kj}^h = -\psi_{jk}^v$;\\
$\psi_{jm}^v = \psi_{jk}^v + \psi_{km}^v$, & & $\psi_{jm}^v = \psi_{jk}^h + \psi_{km}^h$;
\end{tabular}
\end{center}
and therefore $\sim$ is an equivalence relation.

Let us denote $\B = \B' / \sim$ and define the following partial composition laws 
on $\B$
\begin{align*}
\left\{
\begin{matrix}
[K, F, \lambda_{01, 01}] & [L, G, \lambda_{01, 12}] 
\end{matrix}
\right\}
&= [K + b(F)\cdot L + \tau_{\lambda}(L, G), \left\{ F \; G \right\}, \lambda_{01, 02}], \\
\left\{
\begin{matrix}
[K, F, \lambda_{01, 01}] \\
[L, G, \lambda_{01, 12}] 
\end{matrix}
\right\}
&= \left[ K + l(F)^{-1}\cdot L + \sigma_{\lambda}
\left(
\begin{matrix}
F \\
G
\end{matrix}
\right), \left\{
\begin{matrix} F \\
G 
\end{matrix}
\right\}, \lambda_{01, 02} \right].
\end{align*}
It is easy to check that these operations, joint with the quotient topology, endow $\B$
with the structure of a double topological groupoid, that is an extension of $\F$ by $\Kc$, 
and that the canonical projection of $\B$ over $\F$ admit a section over each open set $U^{11}_i$ in the open cover $\Uc_{11}$ of $\F$, \emph{i.e} this extension is an element of 
$\Opext_{\Uc} (\F, \Kc)$. Finally, it is no difficult to prove that this correspondence defines an isomorphism between $\Ho^{1}_{\Tot}(\F, \Kc)$ and $\Opext_{\Uc} (\F, \Kc)$ and then, passing to the limit, it follows 
$$\Ext(\F, \Kc) \cong \check{\Ho}^{1}_{\Tot}(\dbs{\Uc}; \ubs{\A}).$$
This last step finish the proof.


%
\end{proof}

\begin{section}{Acknowledgements}
The first named author was supported by the research project ``Groupoid extensions and cohomology'', ID-PRJ: 00006538 of the Faculty of Sciences of Pontificia Universidad Javeriana, Bogotá, Colombia.

The second named author was partially supported by CONICET, ANPCyT and Secyt (UNC).

\end{section}

\end{document}